\documentclass[letterpaper,12pt]{amsart}

\usepackage[T3,T1]{fontenc}
\DeclareSymbolFont{tipa}{T3}{cmr}{m}{n}
\DeclareMathAccent{\invbreve}{\mathalpha}{tipa}{16}

\usepackage{subfig}
\usepackage{graphicx}
\usepackage{multirow}      
\usepackage{multicol}   
\usepackage{subfig}
\usepackage{booktabs}

\usepackage{amsthm}
\newtheorem{theorem}{Theorem}[section]
\newtheorem{lemma}[theorem]{Lemma}

\newtheorem{remark}[theorem]{Remark}

\usepackage{amsmath}

\usepackage{amsfonts}   
\usepackage{amssymb}    
\usepackage{stmaryrd}   
\usepackage{mathrsfs}   
\usepackage{textcomp}   
\usepackage{yfonts}     
\usepackage{mathtext}   

\usepackage{commath}

\usepackage{color}
\usepackage[dvipsnames]{xcolor}

\usepackage{enumerate}

\newcommand{\diam}{\operatorname{diam}}

\makeatletter
\newcommand\suchthat{
 \@ifstar
  {\mathrel{}\middle|\mathrel{}}
  {\mid}
}
\makeatother

\newcommand{\vol}{\operatorname{vol}}
\newcommand{\trace}{\operatorname{tr}}
\newcommand{\onehalf}{\frac{1}{2}}

\usepackage{nicefrac}

\newcommand{\parametric}{\hat}
\newcommand{\physical}{\invbreve}

\usepackage{tikz}
\usetikzlibrary{calc}

\newcommand{\bbN}{{\mathbb N}}
\newcommand{\bbR}{{\mathbb R}}
\newcommand{\bfg}{{\mathbf g}}
\newcommand{\calA}{{\mathcal A}}
\newcommand{\calB}{{\mathcal B}}
\newcommand{\calF}{{\mathcal F}}
\newcommand{\calP}{{\mathcal P}}
\newcommand{\calT}{{\mathcal T}}
\newcommand{\frakL}{{\mathfrak L}}

\usepackage{hyperref}

\begin{document}

\title
[Transformed FEM]
{Geometric Transformation\\of Finite Element Methods:\\Theory and Applications}

\author{Michael Holst}

\address{UCSD Department of Mathematics, 9500 Gilman Drive MC0112, La Jolla, CA 92093-0112}

\email{mholst@ucsd.edu}

\thanks{MH was supported in part by NSF DMS/FRG Award 1262982 and NSF DMS/CM Award 1620366}

\author{Martin Licht}

\address{UCSD Department of Mathematics, 9500 Gilman Drive MC0112, La Jolla, CA 92093-0112}

\email{mlicht@ucsd.edu}

\thanks{ML was supported in part by NSF DMS/RTG Award 1345013 and DMS/CM Award 1262982}

\subjclass[2000]{}


\begin{abstract}
 We present a new technique to apply finite element methods 
 to partial differential equations over curved domains. 
 A change of variables along a coordinate transformation 
 satisfying only low regularity assumptions 
 can translate a Poisson problem over a curved physical domain
 to a Poisson problem over a polyhedral parametric domain. 
 This greatly simplifies both the geometric setting and the practical 
 implementation, at the cost of having globally rough non-trivial 
 coefficients and data in the parametric Poisson problem. 
 Our main result is that a recently developed broken Bramble-Hilbert lemma 
 is key in harnessing regularity in the physical problem 
 to prove higher-order finite element convergence rates for the 
 parametric problem. 
 Numerical experiments are given which confirm the predictions of our theory.
\end{abstract}

\maketitle

\section{Introduction}
\label{sec:introduction}

The computational theory of partial differential equations 
has been in a paradoxical situation from its very inception:
partial differential equations over domains with curved boundaries 
are of theoretical and practical interest,
but numerical methods are generally conceived only for polyhedral domains. 
Overcoming this geometric gap continues to inspire much ongoing research
in computational mathematics for treating partial differential equations with
geometric features.
Computational methods for partial differential equations over curved domains
commonly adhere to the philosophy of approximating the \emph{physical domain}
of the partial differential equation by a \emph{parametric domain}. 
We mention isoparametric finite element methods \cite{bernardi1989optimal},
surface finite element methods \cite{dziuk2007finite,demlow2009higher,bonito2013afem},
or isogeometric analysis \cite{hughes2005isogeometric} as examples, 
which describe the parametric domain by a polyhedral mesh 
whose cells are piecewise distorted to approximate the physical domain closely or exactly.
\\

In this contribution we approach the topic from a different point-of-view.
Our technique assumes to explicitly know a transformation of the physical domain, 
on which the original partial differential equation is stated, 
onto a polyhedral parametric domain. For example, the unit ball is a domain with curved 
boundary that is homeomorphic to the unit cube. 
Under mild regularity assumptions on that transformation,
the partial differential equation on the curved physical domain 
can be transformed to an equivalent partial differential equation  
over the polyhedral parametric domain. 
It remains then to numerically solve a partial differential equation over the parametric domain.
We believe that the simplification of the geometry is of practical appeal; 
the trade-off is the coordinate transformation contributing low regularity 
terms to the parametric coefficients and parametric right-hand side. 
In a finite element method, 
the effect of using merely approximate problem data 
can be controlled easily by Strang's lemma. 
We emphasize that the transformation from the physical problem
to the parametric problem takes places prior to any numerical analysis. 

In this article, we give a thorough exposition of this technique
and present exemplary numerical computations. 
A mathematical challenge is the regularity of the geometric transformation:
in practice, the transformation is a diffeomorphism \emph{locally} on each cell 
but of rather low regularity \emph{globally} on the entire domain, 
typically no more than bi-Lipschitz.
Thus it is not immediately evident 
how to leverage any higher regularity 
of the original physical problem for quasi-optimal error estimates 
in the finite element method for the practical parametric problem. 
Our main finding is how to overcome that obstacle via a broken Bramble-Hilbert lemma 
that has risen to prominence only recently \cite{veeser2016approximating,camacho2014L2};
we believe the recent nature of the result
to be the reason why our ostensibly simple technique,
at its core only involving a change of variables, 
has not been received earlier by theoretical numerical analysts. 
Thus a second purpose of our article is to advertise 
the broken Bramble-Hilbert lemma to a broader audience.
(In a separate manuscript~\cite{GHL18a}, we develop 
a generalization of the broken Bramble-Hilbert lemma suitable for
use with the finite element exterior calculus; an application of this result
appears in~\cite{GaHo18a}.)

As evidence that the approach has substantial potential for applications,
we note that it has been fairly easy to implement our ideas in the finite 
element software library \texttt{FEniCS} \cite{AlnaesBlechta2015a},
so we believe that practitioners can easily adopt this article's technique. 
Our numerical experiments confirm the theoretically predicted convergence rates.

\smallskip

We now finish this brief introduction by outlining the larger context 
and pointing out some further possibilities of our research. 
This work is a stepping stone towards
developing \emph{intrinsic} finite element methods 
for partial differential equations over manifolds, where it may be
inconvenient or infeasible to describe the manifold extrinsically using
a larger embedding manifold, so that one must work with an intrinsic
description.
If the manifold is computationally represented by a collection
of coordinate charts onto parametric domains, 
then this article lays the foundation 
for the a priori error analysis of a finite element method.
This research agenda will also touch upon 
finite element methods over embedded surfaces \cite{camacho2014L2}.

We are not aware of a prior discussion of this transformation technique in 
the literature of theoretical numerical analysis,
but concrete applications are well-established in computational physics. 
An example is what is known as \textit{cubed sphere} 
in atmospheric and seismic modeling \cite{ronchi1996cubed,ranvcic1996global}. 
We hope that our work helps to connect those developments 
in practical computational physics with the numerical analysis of finite element methods.

For our numerical experiments we have calculated the parametric coefficients manually,
which is feasible in applications such as atmospheric modeling with a fixed geometry of interest. 
However, these calculations can be automated when the transformations are restricted to more specific classes, 
which seems to be more conforming to the demands on numerical methods in engineering. 
For example, our contribution complements the \emph{parametric finite elements} \cite{Zulian2017}
that have recently been formalized in computational engineering, 
albeit without a formal error analysis. 
Moreover, 
our results enable a priori error estimates for \emph{NURBS-enhanced finite element methods} \cite{sevilla2008nurbs},
where the physical geometry is represented over a reference geometry 
in terms of non-uniform rational B-splines (NURBS, \cite{hughes2005isogeometric}). 
Another area of application that we envision are rigorous error estimates 
for simplified computational models in physical modeling over unstructured meshes. 
\\

The remainder of this work is structured as follows.
In Section~\ref{sec:theory}
we introduce our model problem and review the relevant aspects of Galerkin theory.
In Section~\ref{sec:fem} we prove the broken Bramble-Hilbert lemma 
and in Section~\ref{sec:apriori} we elaborate on the a priori error analysis.
Finally, we discuss numerical results in Section~\ref{sec:examples}.

\section{Model Problem and Abstract Galerkin Theory}
\label{sec:theory}

As a model problem throughout this article we consider a variant of the Poisson equation
with a diffusion tensor of low regularity. We then outline an abstract Galerkin theory 
that includes variational crimes. 

\subsection{Function Spaces}
Let $\Omega \subseteq \bbR^{n}$ be a domain.
For $p \in [1,\infty]$ we let $L^{p}(\Omega)$
be the Banach space of $p$-integrable functions over $\Omega$.
Moreover, for $p \in [1,\infty]$ and $s \in \bbR_{0}^{+}$
we let $W^{s,p}(\Omega)$ denote the Sobolev-Slobodeckij space over $\Omega$
with regularity index $s$ and integrability index $p$.
We write $\|\cdot\|_{W^{s,p}(\Omega)}$ for the norm of $W^{s,p}(\Omega)$,
and we write $|\cdot|_{W^{s,p}(\Omega)}$ for the associated seminorm. 
In the case $p=2$ the space $W^{s,p}(\Omega)$ 
carries a Hilbert space structure,
and we denote by $\langle \cdot,\cdot\rangle_{W^{s,2}(\Omega)}$
the Sobolev-Slobodeckij semiscalar product of order $s$. 

Whenever $\Gamma \subseteq \partial\Omega$ is a closed subset of the domain boundary,
we define the space $W^{s,p}(\Omega,\Gamma)$ 
as the closure of the smooth functions in $W^{s,p}(\Omega)$ 
that vanish in an open neighborhood of $\Gamma$. 
For every boundary part $\Gamma \subseteq \partial\Omega$ 
we let $\Gamma^{c}$ denote the \emph{complementary boundary part},
which we define as the closure of $\partial\Omega \setminus \Gamma$. 
Then we write 
$W^{-s,p}(\Omega,\Gamma^{c}) := W^{s,p}(\Omega,\Gamma)^{\ast}$ 
for the dual space of $W^{s,p}(\Omega,\Gamma)$.
This is a Banach space in its own right.

\subsection{Physical Model Problem}
We now introduce the physical model problem. 
Let $\physical\Omega \subseteq \bbR^{n}$ be a Lipschitz domain 
and $\physical\Gamma_{\rm D} \subseteq \partial\physical\Omega$ be closed and non-empty.
We let $\physical\Gamma_{\rm N} \subseteq \partial\physical\Omega$ be the complementary boundary part.
For simplicity we write 
\begin{gather*}
 W^{ s,p}_{\rm D}(\physical\Omega) := W^{s,p}(\physical\Omega,\physical\Gamma_{\rm D}), 
 \quad 
 W^{-s,p}_{\rm N}(\physical\Omega) := W^{s,p}(\physical\Omega,\physical\Gamma_{\rm D})^{\ast}.
\end{gather*}
We assume that $\physical A \in L^{\infty}(\physical\Omega)^{n\times n}$ 
is an essentially symmetric matrix field over $\physical\Omega$ 
that is invertible almost everywhere with 
$\physical A^{-1} \in L^{\infty}(\physical\Omega)^{n\times n}$.  
We write $\|\cdot\|_{L^{2}(\physical\Omega,\physical A)}$ 
for the associated weighted $L^{2}$ norm on the Hilbert space $L^{2}(\physical\Omega)^{n}$,
which is equivalent to the usual $L^{2}$ norm. 

We introduce the symmetric bilinear form of the Poisson problem:  
\begin{gather}
 \physical B : W^{1,2}_{\rm D}(\physical\Omega) \times W^{1,2}_{\rm D}(\physical\Omega) \rightarrow \bbR,
 \quad 
 (\physical u,\physical v) 
 \mapsto 
 \int_{\physical \Omega} \nabla \physical u \cdot \physical A \nabla \physical v \dif \physical x.
\end{gather}
Given a functional $\physical F \in W^{-1,2}_{\rm N}(\physical\Omega)$, 
the model problem is to find $\physical u \in W^{1,2}_{\rm D}(\physical\Omega)$ with 
\begin{gather}
 \label{math:modelproblem}
 \physical B( \physical u , \physical v ) = \physical F( \physical v ),
 \quad 
 \physical v \in W^{1,2}_{\rm D}(\physical\Omega). 
\end{gather}
We recall that 
there exists ${\physical c_{P}} > 0$, 
depending only on $\physical\Omega$, $\physical \Gamma_{\rm D}$, and $\physical A$, such that 
\begin{gather*}
 {\physical c_{P}} 
 \| \physical v \|_{W^{1,2}(\physical\Omega)}^{2} 
 \leq 
 \physical B( \physical v, \physical v ), 
 \quad 
 \physical v \in W^{1,2}_{\rm D}(\physical\Omega).
\end{gather*}
The Lax-Milgram lemma \cite{braess2007finite} thus implies that 
the model problem \eqref{math:modelproblem} has a unique solution 
$\physical u \in W^{1,2}_{\rm D}(\physical\Omega)$
satisfying the stability estimate 
\begin{gather}
 \label{math:discretestability:physical}
 {\physical c_{P}}
 \| \physical u \|_{W^{1,2}_{\rm D}(\physical\Omega)} 
 \leq
 \| \physical F \|_{W^{-1,2}_{\rm N}(\physical\Omega)}
 .
\end{gather}

\subsection{Domain Transformation}
We henceforth call the domain $\physical\Omega \subseteq \bbR^{n}$
the \emph{physical domain}. 
Additionally we now assume to be given another domain $\parametric\Omega \subseteq \bbR^{n}$,
henceforth called \emph{parametric domain}, 
and a homeomorphism 
\begin{gather}
 \label{math:geometrictransformation}
 \Phi : \parametric\Omega \rightarrow \physical\Omega
\end{gather}
from the parametric domain onto the physical domain. 
As a minimal regularity assumption on this homeomorphism 
we assume that 
\begin{gather}
 \label{math:regularity_of_transformation}
 \Phi_{i} \in W^{1,\infty}(\parametric\Omega)
 ,
 \quad 
 \Phi^{-1}_{i} \in W^{1,\infty}(\physical\Omega)
\end{gather}
for each coordinate index $1 \leq i \leq n$.  
This regularity assumption is satisfied, for example,
in the special case that $\Phi$ is bi-Lipschitz.
We write 
\begin{gather*}
 \parametric\Gamma_{\rm D} = \Phi^{-1} ( \physical\Gamma_{\rm D} ),
 \quad 
 \parametric\Gamma_{\rm N} = \Phi^{-1} ( \physical\Gamma_{\rm N} ) 
\end{gather*}
for the corresponding boundary patches along the parametric domain.
On the parametric domain, too, we introduce the short-hand notation 
\begin{gather*}
 W^{ s,p}_{\rm D}(\parametric\Omega) := W^{s,p}(\parametric\Omega,\parametric\Gamma_{\rm D}), 
 \quad 
 W^{-s,p}_{\rm N}(\parametric\Omega) := W^{s,p}(\parametric\Omega,\parametric\Gamma)^{\ast}.
\end{gather*}
The homeomorphism $\Phi$ and its inverse $\Phi^{-\ast}$ define isomorphisms 
between Sobolev spaces on the parametric domain and the physical domain:
\begin{subequations}
\label{math:pullback}
\begin{gather}
 \label{math:pullback:onto_parameter_domain}
 \Phi^{ \ast} : W^{1,2}_{\rm D}(\physical\Omega) \rightarrow W^{1,2}_{\rm D}(\parametric\Omega),
 \quad 
 \physical v \mapsto \physical v \circ \Phi
 ,
 \\
 \label{math:pullback:onto_physical_domain}
 \Phi^{-\ast} : W^{1,2}_{\rm D}(\parametric\Omega) \rightarrow W^{1,2}_{\rm D}(\physical\Omega),
 \quad 
 \parametric v \mapsto \parametric v \circ \Phi^{-1}
 . 
\end{gather}
\end{subequations}

\subsection{Parametric and Physical Model Problem}
The model problem over the physical domain 
is equivalent to a variational problem of the same class over the parametric domain. 
To begin with, we call the matrix field $\physical A : \physical\Omega \rightarrow \bbR^{n\times n}$
the \emph{physical coefficient} 
and introduce the corresponding \emph{parametric coefficient} as 
\begin{gather}
 \label{math:parametric_diffusion_tensor}
 \parametric A : \parametric\Omega \rightarrow \bbR^{n \times n},
 \quad 
 \parametric x
 \mapsto 
 \left| \det \Dif \Phi \right|_{|\parametric x} 
 \cdot 
 \Dif\Phi^{-1}_{|\Phi\left( \parametric x \right)} 
 \physical A_{|\Phi\left( \parametric x \right)}
 \Dif\Phi^{-t}_{|\Phi\left( \parametric x \right)} 
 .
\end{gather}
Next we define the \emph{parametric bilinear form} 
\begin{gather}
 \label{math:parametric_bilinear}
 \parametric B : 
 W^{1,2}_{\rm D}(\parametric \Omega) \times W^{1,2}_{\rm D}(\parametric \Omega) \rightarrow \bbR,
 \quad 
 (\parametric u,\parametric v) 
 \mapsto 
 \int_{\parametric\Omega} 
 \nabla \parametric u 
 \cdot
 \parametric A
 \nabla \parametric v
 ,
\end{gather}
and the \emph{parametric right-hand side} $\parametric F \in W^{-1,2}_{\rm N}(\parametric\Omega)$
via 
\begin{gather}
 \label{math:parametric_right_hand_side}
 \parametric F( \parametric v ) := \physical F( \Phi^{-\ast} \parametric v ),
 \quad 
 \parametric v \in W^{1,2}_{\rm D}(\parametric\Omega)
 . 
\end{gather}
Note the relation  
\begin{gather*}
 \parametric B( \Phi^{\ast}\physical u , \Phi^{\ast}\physical v )
 =
 \physical B( \physical u , \physical v ),
 \quad 
 \physical u, \physical v \in W^{1,2}_{\rm D}(\physical \Omega)
 .
\end{gather*}
There exists a constant $\parametric c_{P} > 0$,
which we call \emph{parametric coercivity constant}
and which depends only on $\parametric \Omega$, $\parametric \Gamma_{D}$, and $\parametric A$,
that satisfies 
\begin{gather}
 \label{math:coercivity:parametric}
 {\parametric c_{P}} 
 \| \parametric v \|_{W^{1,2}(\parametric\Omega)}^{2} 
 \leq 
 \parametric B( \parametric v, \parametric v ), 
 \quad 
 \parametric v \in W^{1,2}_{\rm D}(\parametric\Omega).
\end{gather}
Such a constant can also be bounded in terms of $\physical c_{P}$ 
and the derivatives of $\Phi$ and $\Phi^{-1}$ up to first order:
\begin{gather*}
 {\parametric c_{P}} 
 \leq 
 \| \det\Dif\Phi \|_{L^{\infty}(\parametric\Omega)}
 \| \Dif\Phi^{-1} \|_{L^{\infty}(\parametric\Omega)}^{2}
 {\physical c_{P}} 
\end{gather*}
The parametric model problem is finding 
$\parametric u \in W^{-1,2}_{\rm N}(\parametric\Omega)$
such that 
\begin{gather}
 \label{math:trafo:modelproblem}
 \parametric B( \parametric u , \parametric v ) = \parametric F( \parametric v ),
 \quad 
 \parametric v \in W^{1,2}_{\rm D}(\parametric \Omega). 
\end{gather}
The unique solution 
$\physical u \in W^{1,2}_{\rm D}(\physical\Omega)$
of the physical model problem \eqref{math:modelproblem}
and the unique solution 
$\parametric u \in W^{1,2}_{\rm D}(\parametric \Omega)$ 
of the parametric model problem \eqref{math:trafo:modelproblem}
satisfy $\parametric u = \Phi^{\ast} \physical u$. 
We henceforth call \eqref{math:modelproblem} the \emph{physical model problem}
and \eqref{math:trafo:modelproblem} the \emph{parametric model problem}. 
We have
\begin{gather}
 \label{math:discretestability:parametric}
 {\parametric c_{P}} \| \parametric u \|_{W^{1,2}(\parametric\Omega)} \leq \| \parametric F \|_{W^{-1,2}_{N}(\parametric\Omega)}.
\end{gather}
We consider the transformation of the physical right-hand side $\physical F$ in more detail. 
The physical right-hand side $\physical F$ can be represented 
by a scalar function $\physical f \in L^{2}(\physical\Omega)$ 
and a vector field $\physical \bfg \in L^{2}(\physical\Omega)^{n}$
such that 
\begin{gather}
 \label{math:rhs_representation:physical}
 \physical F( \physical v ) 
 = 
 \int_{\physical\Omega} \physical f \physical v \dif \physical x
 + 
 \int_{\physical \Omega} \physical \bfg \cdot \nabla \physical v \dif \physical x,
 \quad 
 \physical v \in W^{1,2}_{\rm D}(\physical\Omega).
\end{gather}
The parametric right-hand side $\parametric F$ is then represented as follows: 
for every $\parametric v \in W^{1,2}_{\rm D}(\parametric\Omega)$ we have 
\begin{align}
 \label{math:rhs_representation:parametric}
 \parametric F( \parametric v )
 &= 
 \int_{\parametric\Omega} 
 \left| \det \Dif \Phi \right|
 ( \physical f \circ \Phi )
 \parametric v
 \dif \parametric x
 +
 \int_{\parametric\Omega} 
 \left| \det \Dif \Phi \right|
 \left(
 \Dif\Phi^{-1}_{|\Phi}
 ( \physical\bfg_{|\Phi} )
 \right)
 \nabla \parametric v 
 \dif \parametric x
 . 
\end{align}

\begin{remark}
  Representation \eqref{math:rhs_representation:physical}
  is not only a theoretical consequence of the Riesz representation theorem;
  it appears practically when encoding boundary conditions in the problem data. 
  Suppose that we search $\physical u_{\ast} \in W^{1,2}(\physical\Omega)$ 
  such that distributionally $-\operatorname{div}(\physical A\nabla \physical u_{\ast}) = \physical f$,
  where $\physical f \in L^{2}(\physical\Omega)$,
  and satisfying the following mixed boundary conditions:
  along the Dirichlet boundary part $\physical \Gamma_D$,
  we want $\physical u_{\ast}$ to have the same boundary traces as some function $\physical w \in W^{1,2}(\physical\Omega)$,
  and along the Neumann boundary part $\physical \Gamma_N$,
  we want $\physical u_{\ast}$ to the same normal boundary trace as some vector field $\physical \bfg \in L^{2}(\physical\Omega)$
  with $\operatorname{div} \physical \bfg \in L^{2}(\physical\Omega)$. 
  It is easily seen that if $\physical u \in W^{1,2}_{\rm D}(\physical\Omega)$
  satisfies the model problem with right-hand side 
  \begin{gather}
  \label{math:rhs_representation:physical:bc}
  \int_{\physical \Omega} 
  \physical f \physical v
  \dif \physical x
  - 
  \int_{\physical \Omega} 
  \nabla \physical w \cdot \physical A \nabla \physical v 
  \dif \physical x
  +
  \int_{\physical \Omega} 
  ( \operatorname{div} \physical \bfg ) \physical v
  +
  \physical \bfg \cdot \nabla \physical v
  \dif \physical x,
  \quad 
  \physical v \in W^{1,2}(\physical\Omega)
  ,
  \end{gather}
  then such $\physical u_{\ast}$ is found by $\physical u_{\ast} = \physical w + \physical u$. 
  The corresponding right-hand side of the parametric model problem is 
  \begin{gather}
  \label{math:rhs_representation:parametric:bc}
  \int_{\parametric \Omega} 
  \parametric f \parametric v
  \dif \parametric x
  - 
  \int_{\parametric \Omega} 
  \nabla \parametric w \cdot \parametric A \nabla \parametric v 
  \dif \parametric x
  +
  \int_{\parametric \Omega} 
  ( \operatorname{div} \parametric \bfg ) \parametric v
  +
  \parametric \bfg \cdot \nabla \parametric v
  \dif \parametric x,
  \quad 
  \parametric v \in W^{1,2}(\parametric \Omega)
  ,
  \end{gather}
  where 
  \begin{gather}
  \label{math:rhs_representation:parametric:bcdata}
  \parametric f 
  = 
  \left| \det \Dif \Phi \right|
  ( \physical f \circ \Phi ),
  \quad 
  \parametric w
  = 
  \physical w \circ \Phi, 
  \quad
  \parametric{\bfg}
  =
  \left| \det \Dif \Phi \right|
  \cdot
  \Dif\Phi^{-1}_{|\Phi} 
  ( \physical \bfg_{|\Phi} )
  .
  \end{gather}
  We note that $\parametric \bfg$ is the Piola transformation of the vector field $\physical \bfg$, 
  which is known to preserve the class of divergence-conforming square-integrable vector fields. 
  We note that the parametric right-hand side has the same structure as the physical right-hand side,
  encoding the same type of boundary conditions.
\end{remark}

\subsection{Galerkin Theory}
We review conforming and non-conforming 
Galerkin approximation theories for the parametric model problem.
We assume that we have a closed subspace $\parametric V_{h} \subseteq W^{1,2}_{\rm D}(\parametric\Omega)$. 
A conforming Galerkin approximation for the model problem \eqref{math:modelproblem} 
seeks a solution $\parametric u_{h} \in \parametric V_{h}$ to 
\begin{gather}
 \label{math:discreteproblem:conforming}
 \parametric B( \parametric u_{h} , \parametric v_{h} ) = \parametric F( \parametric v_{h} ),
 \quad 
 \parametric v_{h} \in \parametric V_{h}. 
\end{gather}
As in the case of the original problem,
the Lax-Milgram lemma gives a unique solution $\parametric u_{h} \in \parametric V_{h}$
to the discrete problem \eqref{math:discreteproblem:conforming}, 
and we have 
\begin{gather*}
 {\parametric c_{P}}
 \| \parametric u_{h} \|_{W^{1,2}_{\rm D}(\parametric\Omega)} 
 \leq
 \| \parametric F \|_{W^{-1,2}_{\rm N}(\parametric\Omega)}
 .
\end{gather*}
In many applications, 
the bilinear form of the model problem or the right-hand side functional 
cannot be evaluated exactly but merely approximately over the Galerkin space.
Formalizing, 
we assume to have another bounded bilinear form 
\begin{gather}
 \parametric B_{h} : \parametric V_{h} \times \parametric V_{h} \rightarrow \bbR, 
\end{gather}
which is ought to approximate the original bilinear form $\parametric B$
over the Galerkin space $\parametric V_{h}$.
We consider the following problem:
given an approximate right-hand side functional $\parametric F_{h} \in \parametric V_{h}^{\ast}$, 
we seek a solution ${\underline{\parametric u}}_{h} \in \parametric V_{h}$ of 
\begin{gather}
 \label{math:discreteproblem:nonconforming}
 \parametric B_{h}( {\underline{\parametric u}}_{h} , \parametric v_{h} ) = \parametric F_{h}( \parametric v_{h} )
 ,
 \quad 
 \parametric v_{h} \in \parametric V_{h}
 . 
\end{gather}
It is practically reasonable to assume 
the existence of a \emph{discrete parametric coercivity constant} ${\parametric c_{P,h}} > 0$ such that 
\begin{gather}
 \label{math:coercivity:discrete}
 {\parametric c_{P,h}} \| \parametric v_{h} \|_{W^{1,2}(\parametric\Omega)} 
 \leq 
 \parametric B_{h}( \parametric v_{h}, \parametric v_{h} ), 
 \quad 
 \parametric v_{h} \in \parametric V_{h}.
\end{gather}
Under this assumption, 
we can again apply the Lax-Milgram lemma 
to establish the well-posedness and stability of the Galerkin method.
There exists a unique solution ${\underline{\parametric u}}_{h} \in \parametric V_{h}$ 
to the non-conforming discrete problem such that 
\begin{gather}
 \label{math:discretestability:nonconforming}
 {\parametric c_{P,h}}
 \| {\underline{\parametric u}}_{h} \|_{\parametric V_{h}} 
 \leq
 \| \parametric F_{h} \|_{\parametric V_{h}^{\ast}}
 .
\end{gather}
We discuss a priori error estimates after an excursion 
into finite element approximation theory in the next section. 

\begin{remark}
 Any conforming Galerkin method for the parametric model problem 
 over $\parametric V_{h} \subseteq W^{1,2}_{\rm D}(\parametric \Omega)$
 translates into a conforming Galerkin method for the physical model problem
 over $\physical V_{h} \subseteq W^{1,2}_{\rm D}(\physical \Omega)$. 
 The subspaces are related through $\parametric V_{h} = \Phi^{\ast} \physical V_{h}$. 
\end{remark}

\section{Finite Element Spaces and Error Estimates}
\label{sec:fem}

This section introduces finite element spaces for our model problem 
and proves an approximation result central to our numerical approach.
We continue the geometric setup of the preceding section
but introduce a triangulation of the parametric domain as additional structure. 

\subsection{Simplices and Triangulations}
We commence with gathering a few definitions concerning simplices and triangulations 
that will be used below.

A non-empty set $T \subseteq \bbR^{n}$ is a \emph{$d$-dimensional simplex} 
if it is the convex closure of $d+1$ affinely independent points $x_{0},\dots,x_{d}$,
which are called the \emph{vertices} of the simplex. 
For any $d$-dimensional simplex $T$ we let $\calF(T)$
denote the set of its $d+1$ facets, 
where a facet is to be understood as a subsimplex of $T$
whose vertices are all but one of the $d+1$ vertices of $T$. 

For the purpose of this article,
a simplicial complex is a collection $\calT$ of simplices such that 
for all $T \in \calT$ and all $S \in \calF(T)$ we have $S \in \calT$
and such that for all $T_1,T_2 \in \calT$ the intersection $T_1 \cap T_2$
is either empty or a simplex whose vertices are vertices of both $T_1$ and $T_2$. 

For any simplex $T$ of positive dimension $d$ we let 
$h_T$ be its diameter, and we call $\mu(T) := \diam(T)^{d} / \vol^{d}(T)$ the \emph{shape measure} of $T$. 
The shape measure of a simplicial complex $\calT$ is the maximum 
of the shape measures of its simplices and is denoted by $\mu(\calT)$.
We also let $h_{\calT}$ be the maximum diameter of any simplex of $\calT$. 

Following \cite{veeser2016approximating}, 
we call a finite simplicial complex $\calT$ \emph{face-connected}
whenever the following condition is true:
for all $n$-dimensional simplices $S,T \in \calT$
with $S \cap T \neq \emptyset$,
there exists a sequence $T_0,T_1,\dots,T_N$ of $n$-dimensional simplices of $\calT$
such that $T_0 = S$ and $T_N = T$ and such that 
for all $1 \leq i \leq N$ we have that $F_{i} := T_{i} \cap T_{i-1}$
satisfies $F_{i} \in \calF(T_{i-1}) \cap \calF(T_{i})$ and $S \cap T \subseteq F_{i}$. 
In other words, 
whenever two $n$-dimensional simplices share a common subsimplex, 
then we can traverse from the first to the second simplex 
by crossing facets of adjacent $n$-dimensional simplices 
and such that every simplex during the traversal 
will contain the intersection of the original two simplices as a subset.

\subsection{Polynomial Approximation over a Simplex}

Whenever $T$ is a simplex of dimension $d$,
we let $\calP_{r}(T)$ denote the polynomials over $T$ of degree at most $r \in \bbN_{0}$. 
We study interpolation and projection operators onto the polynomials of a simplex.

We first recall the definition of the Lagrange points over the simplex $T$.
Letting $x_{0},\dots,x_{d}$ denote the vertices of the simplex $T$,
we define the set of degree $r$ \emph{Lagrange points} by 
\begin{gather*}
 \frakL_{r}(T) 
 := 
 \left\{\; 
  \left( \alpha_{0} x_{0} + \dots + \alpha_{d} x_{d} \right) / r
  \; \middle| \;
  \alpha = ( \alpha_0,\dots,\alpha_d ) \in \bbN_{0}^{d+1}, \; |\alpha| = r
 \;\right\}
 .
\end{gather*}
We note that $\frakL_{r}(F) \subseteq \frakL_{r}(T)$ for every $F \in \calF(T)$.
We distinguish inner and outer Lagrange points: 
we let $\partial\frakL_{r}(T) \subseteq \frakL_{r}(T)$ 
be the set of Lagrange points of $T$ that lie on the boundary of $T$
and let $\mathring\frakL_{r}(T) = \frakL_{r}(T) \setminus \partial\frakL_{r}(T)$.

For every $x \in \frakL_{r}(T)$ we let $\delta_{x}^{T}$ denote the Dirac delta
associated to that Lagrange point, which is an element of the dual space
of $\calP_{r}(T)$. The Dirac deltas associated to the Lagrange points 
constitute a basis for the dual space of $\calP_{r}(T)$.
The Lagrange polynomials are the associated predual basis:
for every $x \in \frakL_{r}(T)$ 
we let the polynomial $\Phi^{T}_{r,x} \in \calP_{r}(T)$
be defined uniquely by 
\begin{gather*}
 \Phi^{T}_{r,x}(y) = \delta^{T}_{y} \Phi^{T}_{r,x} = \delta_{xy},
 \quad 
 y \in \frakL_{r}(T),
\end{gather*}
where $\delta_{xy}$ denotes the Kronecker delta. Obviously,
\begin{gather}
 \label{math:polynomialreproduction}
 v = \sum_{ x \in \frakL_{r}(T) } (\delta^{T}_{x} v ) \Phi^{T}_{r,x},
 \quad 
 v \in \calP_{r}(T).
\end{gather}
It is worthwhile to extend the domain of the Dirac Deltas
onto $L^{p}(T)$. For each Lagrange node $x \in \frakL_{r}(T)$
we uniquely define $\Psi^{T}_{r,x} \in \calP_{r}(T)$ through the condition 
\begin{gather*}
 \int_{T} \Psi^{T}_{r,x} v \dif x = \delta^{T}_{x} v,
 \quad 
 v \in \calP_{r}(T),
\end{gather*}
or equivalently,
\begin{gather*}
 \int_{T} \Psi^{T}_{r,x} \Phi^{T}_{r,y} \dif x = \delta_{xy},
 \quad 
 y \in \frakL_{r}(T).
\end{gather*}
The following estimates derive from a scaling argument
and is stated without proof: 

\begin{lemma}
 \label{prop:referencefunctions}
 Let $T$ be a $d$-dimensional simplex. 
 Let
 $p \in [1,\infty)$, 
 $s \in \bbR$, and $r \in \bbN_{0}$. 
 Then there exists $C_{\mu,d,r,s,p} > 0$,
 depending only on $d$, $r$, $s$, $p$, and $\mu(T)$, 
 such that 
 \begin{gather}
  \label{math:referencefunctions}
  \left| \Phi^{T}_{r,x} \right|_{W^{s,p}(T)}
  \leq
  C_{\mu,d,r,s,p} h_{T}^{\frac{d}{p}-s},
  \quad 
  \left| \Psi^{T}_{r,x} \right|_{W^{s,p}(T)^{\ast}} 
  \leq 
  C_{\mu,d,r,s,p} h_{T}^{-\frac{d}{p}+s}.
 \end{gather}
\end{lemma}

We fix quasi-optimal interpolation operators over each triangle. 
Specifically, for each $n$-dimensional $T \in \calT$ 
we assume to have an idempotent linear mapping 
\begin{gather*}
 P_{T,r,s,p} : W^{s,p}(T) \rightarrow \calP_{r}(T) \subseteq W^{s,p}(T)
\end{gather*}
that satisfies 
\begin{gather*}
 \int_{T} P_{T,r,s,p} v \dif x = \int_{T} v \dif x, 
 \quad 
 v \in W^{s,p}(T),
\end{gather*}
and such that for some $C^{\rm I}_{d,s,r,p,\mu} > 0$,
depending only on $n$, $s$, $p$, $r$, and $\mu(\calT)$, 
we have 
\begin{gather*}
 | u - P_{T,r,s,p} u |_{W^{s,p}(T)}
 \leq 
 C^{\rm I}_{d,s,r,p,\mu}
 \inf_{ v \in W^{s,p}(T) }
 | u - v |_{W^{s,p}(T)}, 
 \quad 
 v \in W^{s,p}(T)
 .
\end{gather*}
This existence of such a mapping follows from a scaling argument.

The following trace inequality will be used.

\begin{lemma}
 \label{prop:traceinequality}
 Let $T$ be a $d$-dimensional simplex and let $F$ be a facet of $T$.
 Let $p \in [1,\infty)$ and $s \in (\nicefrac{1}{p},1]$. 
 Then there exists a constant $C^{\rm Tr}_{p,s,d,\mu} > 0$,
 depending only on $p$, $s$, $d$, and $\mu(T)$, such that 
 \begin{gather*}
  \| v - P_{T,r,s,p} v \|_{L^{p}(F)}
  \leq 
  C^{\rm Tr}_{p,s,d,\mu}
  h^{s - \frac{1}{p}}_{T}
  | v - P_{T,r,s,p} v |_{W^{s,p}(T)},
  \quad 
  v \in W^{s,p}(T)
  .
 \end{gather*}
\end{lemma}

\begin{proof}
 We note that $v - P_{T,r,s,p} v$ has zero average over $T$. 
 The lemma now follows by combining 
 the trace inequality that is Lemma~7.2 of \cite{ern2017finite}
 and 
 the Poincar\'e inequality that is Lemma~7.1 of \cite{ern2017finite}.
\end{proof}

\subsection{Polynomial Approximation over Triangulations}
We now extend our discussion to piecewise polynomial approximation spaces 
over entire triangulations.
We fix a triangulation $\calT$ of the parametric domain $\parametric\Omega$,
i.e., a simplicial complex the union of whose simplices is the closure of the parametric domain.
In order to formally handle boundary conditions,
we assume that the boundary part $\parametric\Gamma_{\rm D}$
is the union of simplices in $\calT$.

We first introduce the \emph{broken} or \emph{non-conforming} Lagrange space 
\begin{gather*}
 \calP_{r,-1}(\calT)
 :=
 \left\{ 
  u \in L^{1}(\parametric\Omega)
  \; \middle| \; 
  \forall T \in \calT : u_{|T} \in \calP_{r}(T)
 \right\}
 .
\end{gather*}
The \emph{conforming} Lagrange spaces without and with boundary conditions are 
\begin{gather*}
 \calP_{r}(\calT)
 :=
 \calP_{r,-1}(\calT)
 \cap 
 W^{1,2}(\parametric\Omega),
 \qquad 
 \calP_{r,\rm D}(\calT)
 :=
 \calP_{r,-1}(\calT)
 \cap 
 W^{1,2}_{\rm D}(\parametric\Omega)
 .
\end{gather*}
We construct a basis for the global finite element space 
from the Lagrange basis functions over single simplices. 
We first introduce the Lagrange points of the triangulation, 
\begin{gather*}
 \frakL_{r}(\calT) := \bigcup_{ T \in \calT } \frakL_{r}(T). 
\end{gather*}
We note that Lagrange points can be shared between distinct simplices.
Extending the notion of Lagrange polynomials to the case of triangulations, 
for each $x \in \frakL_{r}(\calT)$ 
we define the function $\Phi_{r,x}^{\calT} \in \calP_{r}(\calT)$ 
on each simplex $T \in \calT$ via 
\begin{align*}
 \Phi^{\calT}_{r,x|T}
 :=
 \left\{ 
  \begin{array}{ll}
   \Phi_{r,x}^{T} & \text{ if } x \in \frakL_{r}(T), 
   \\
   0            & \text{ if } x \notin \frakL_{r}(T).
  \end{array}
 \right. 
\end{align*}
For the degrees of freedom 
we assume that $r \in \bbN$, $p \in [1,\infty)$ and $s \in (\nicefrac{1}{p},1]$ 
and apply the following construction. 
Whenever $x \in \mathring\frakL_{r}(T)$ 
is an internal Lagrange point of a full-dimensional simplex $T \in \calT$,
then we define 
\begin{gather*}
 J^{\calT}_{r,s,p,x} 
 : 
 W^{s,p}(\parametric\Omega) \rightarrow \bbR, 
 \quad 
 v \mapsto \delta^{T}_{x} P_{T,r,s,p} v
 .
\end{gather*}
Whenever $x \in \frakL_{r}(\calT)$ is not an internal Lagrange point 
of any full-dimensional simplex of the triangulation, 
then we first fix a facet $F_{x} \in \calT$ of codimension one 
for which $x \in \frakL_{r}(F)$ holds;
moreover,
if $x \in \parametric\Gamma_{\rm D}$,
then we require that $F_{x} \subseteq \parametric\Gamma_{\rm D}$.
It is easily seen that this condition can always be satisfied. 
Now we define 
\begin{gather*}
 J^{\calT}_{r,s,p,x} 
 : 
 W^{s,p}(\parametric\Omega) \rightarrow \bbR, 
 \quad 
 \parametric v \mapsto \int_{F_{x}} \Psi^{F}_{r,x} \trace_{F}(\parametric v) \dif x
 .
\end{gather*}
The global projection is 
\begin{gather*}
 \Pi : W^{s,p}(\parametric\Omega) \rightarrow \calP_{r}(\calT),
 \quad 
 \parametric v \mapsto \sum_{ x \in \frakL_{r}(\calT) } J^{\calT}_{r,s,p,x}(\parametric v) \Phi^{\calT}_{r,x}.
\end{gather*}
This operator is idempotent and bounded.
Furthermore, the interpolant preserves boundary conditions:
$\Pi v \in \calP_{r,\rm D}(\calT)$
whenever 
$v \in W^{s,p}(\parametric\Omega,\parametric\Gamma)$.
\\

This completes the construction of a Scott-Zhang-type interpolant. 
We now discuss a general error estimate for this approximation operator. 
A special case of the following result is due to Veeser \cite{veeser2016approximating},
and a slightly different version is due to Camacho and Demlow \cite{camacho2014L2}.

\begin{theorem}
 \label{prop:veeser}
 Assume that $\calT$ is face-connected 
 and
 let $p \in [1,\infty)$ and $s \in \bbR$ with $s > \nicefrac{1}{p}$.
 Then there exist $C_{\Pi} > 0$
 such that 
 for all $r \in \bbN$, 
 all full-dimensional $T \in \calT$, 
 and all $\parametric v \in W^{s,p}(\parametric\Omega)$
 we have 
 \begin{align*}
  &
  \left| \parametric v - \Pi \parametric v \right|_{W^{s,p}(T)} 
  \leq 
  \left| \parametric v - P_{T,r,s,p} \parametric v \right|_{W^{s,p}(T)} 
  +
  C_{\Pi}
  \sum_{ \substack{ T' \in \calT \\ T \cap T' \neq \emptyset } }
  \left| \parametric v - P_{T',r,s,p} \parametric v \right|_{W^{s,p}(T')} 
  . 
 \end{align*}
 The constant $C_{\Pi}$ depends only on $p$, $d$, $r$, $s$, and $\mu(\calT)$.
\end{theorem}

\begin{proof}
 We first observe via the triangle inequality that 
 \begin{gather*}
  \left| \parametric v - \Pi \parametric v \right|_{W^{s,p}(T)} 
  \leq   
  \left| \parametric v - P_{T,r,s,p} \parametric v \right|_{W^{s,p}(T)} 
  +
  \left| P_{T,r,s,p} \parametric v - \Pi \parametric v \right|_{W^{s,p}(T)} 
  .
 \end{gather*}
 The polynomial identity \eqref{math:polynomialreproduction} and Lemma~\ref{prop:referencefunctions} give 
 \begin{align*}
  \left| P_{T,r,s,p} \parametric v - \Pi \parametric v \right|_{W^{s,p}(T)}
  &\leq 
  \sum_{ x \in \frakL_{r}(T) }
  \left| \delta^{T}_{x} P_{T,r,s,p} \parametric v - \delta^{T}_{x} \Pi \parametric v \right| 
  \cdot \left| \Phi^{T}_{r,x} \right|_{W^{s,p}(T)}
  \\&\leq 
  C_{\mu,d,r,s,p}
  h^{\frac{n}{p}-s}_{T}
  \sum_{ x \in \frakL_{r}(T) }
  \left| \delta^{T}_{x} P_{T,r,s,p} \parametric v - \delta^{T}_{x} \Pi \parametric v \right| 
  . 
 \end{align*}
 Suppose that $x \in \mathring\frakL_{r}(T)$ is an internal Lagrange point.
 Then definitions imply 
 \begin{gather*}
  \delta^{T}_{x} P_{T,r,s,p} \parametric v = \delta^{T}_{x} \Pi \parametric v.
 \end{gather*}
 If instead $x \in \partial\frakL_{r}(T)$ is a Lagrange point in the boundary,
 the estimate is more intricate.
 Recall that the facet $F_{x}$ associated to the Lagrange point $x$
 is a facet of some simplex $S \in \calT$ in the triangulation.
 Since $\calT$ is face-connected,
 there exists a sequence $T_0,T_1,\dots,T_N$ of $n$-dimensional simplices of $\calT$
 such that $T_0 = S$ and $T_N = T$ and such that 
 for all $1 \leq i \leq N$ we have that $F_{i} := T_{i} \cap T_{i-1}$
 satisfies $F_{i} \in \calF(T_{i-1}) \cap \calF(T_{i})$ and $T \cap S \subseteq F_{i}$. 
 Furthermore, we may assume that $F_{0} = F_{x}$.
 
 Unrolling definitions and using a telescopic expansion, we see 
 \begin{align*}
  &
  \delta^{T}_{x} P_{T,r,s,p} \parametric v - \delta^{T}_{x} \Pi \parametric v \dif x
  \\&=
  \delta^{T}_{x} P_{T,r,s,p} \parametric v - \int_{F_{x}} \Psi^{F_{x}}_{r,x} \trace_{F_{x}} \parametric v \dif x
  \\&=
  \int_{F_{N}} \Psi^{F_{N}}_{r,x} \trace_{F_{N}} P_{T,r,s,p} \parametric v \dif x
  - 
  \int_{F_{0}} \Psi^{F_{0}}_{r,x} \trace_{F_{0}} \parametric v \dif x
  \\&=
  \int_{F_{0}} 
  \Psi^{F_{0}}_{r,x} \trace_{F_{0}} ( P_{T_{0},r,s,p} \parametric v - \parametric v )
  \dif x
  \\&\qquad\qquad
  +
  \sum_{ i = 1 }^{N}
  \int_{F_{i}} 
  \Psi^{F_{i}}_{r,x} \trace_{F_{i}} ( P_{T_{i},r,s,p} \parametric v - P_{T_{i-1},r,s,p} \parametric v )
  \dif x
  . 
 \end{align*}
 Since $F_{0} \subset T_{0}$ and $P_{T_{0},r,s,p} \parametric v - \parametric v$ 
 has by construction vanishing mean over $T$,
 we use H\"older's inequality and Lemma~\ref{prop:referencefunctions}
 to conclude that 
 \begin{align*}
  \left|
  \int_{F_{0}} \Psi^{F_{0}}_{r,x} \trace_{F_{0}} ( P_{T_{0},r,s,p} \parametric v - \parametric v )
  \dif x
  \right|
  &\leq 
  C_{\mu,d,r,s,p}
  h_{T_{0}}^{\frac{1-n}{p}}
  \left\| \trace_{F_{0}} ( P_{T_{0},r,s,p} \parametric v - \parametric v ) \right\|_{L^{p}(F_{0})} 
  .
 \end{align*}
 Similarly, for $1 \leq i \leq N$ we observe $F_{i} \subseteq T_{i} \cap T_{i-1}$
 and thus get 
 \begin{align*}
  &
  \left| 
   \int_{F_{i}} 
   \Psi^{F_{i}}_{r,x} \trace_{F_{i}} ( P_{T_{i},r,s,p} \parametric v - P_{T_{i-1},r,s,p} \parametric v )
   \dif x
  \right|
  \\&\qquad
  \leq 
  C_{\mu,d,r,s,p}
  h_{T}^{\frac{1-n}{p}}
  \left\| \trace_{F_{i}} ( P_{T_{i  },r,s,p} \parametric v - \parametric v ) \right\|_{L^{p}(F_{i})} 
  \\&\qquad\qquad
  +
  C_{\mu,d,r,s,p}
  h_{T}^{\frac{1-n}{p}}
  \left\| \trace_{F_{i}} ( P_{T_{i-1},r,s,p} \parametric v - \parametric v ) \right\|_{L^{p}(F_{i})} 
  . 
 \end{align*}
 Here we have used Lemma~\ref{prop:referencefunctions}
 and the fact that $\parametric v \in W^{s,p}(\parametric \Omega)$
 has well-defined traces along simplex facets.
 Recall that Lemma~\ref{prop:traceinequality} gives 
 \begin{align*}
  \| \trace_{F} \left( P_{T,r,s,p} \parametric v - \parametric v \right) \|_{L^{p}(F)} 
  &\leq 
  C^{\rm Tr}_{p,s,d,\mu} h^{s-\frac{1}{p}}_{T} 
  \left| P_{T,r,s,p} \parametric v - \parametric v \right|_{W^{s,p}(T)}
 \end{align*}
 whenever $F \subseteq T$ is a facet of a simplex $T \in \calT$.
 Note that the shape measure of $\calT$ bounds the number of simplices of the triangulation
 adjacent to $T$ (see \cite[Lemma~II.4.1]{licht2017priori}). 
 Hence there exists $C_{\Pi} > 0$,
 depending only on $s$, $r$, $d$, and $\mu(\calT)$,
 such that 
 \begin{gather*}
  \left| P_{T,r,s,p} \parametric v - \Pi \parametric v \right|_{W^{s,p}(T)}
  \leq 
  C_{\Pi}
  h_{T}^{\frac{n}{p}-s}
  \sum_{ \substack{ T' \in \calT \\ T \cap T' \neq \emptyset } }
  h_{T}^{\frac{1-n}{p}}
  h^{s-\frac{1}{p}}_{T}
  \left| P_{T',r,s,p} \parametric v - \parametric v \right|_{W^{s,p}(T')}
  .   
 \end{gather*}
 This completes the proof.
\end{proof}

\subsection{Applications in Approximation Theory}
One major application of the approximation theorem 
is to compare the approximation quality 
of conforming and non-conforming finite element spaces. 
One the one hand, for every $\parametric v \in W^{1,p}(\parametric\Omega)$ 
we obviously have 
\begin{align}
 \label{math:confvsdc}
 \inf_{ \parametric w_{h} \in \calP_{r,-1}(\calT) } 
 \left| \parametric v - \parametric w_{h} \right|_{W^{1,p}(\parametric\Omega)}
 &\leq 
 \inf_{ \parametric w_{h} \in \calP_{r, \rm D}(\calT) } 
 \left| \parametric v - \parametric w_{h} \right|_{W^{1,p}(\parametric\Omega)}
 .
\end{align}
Surprisingly, a converse inequality holds 
when assuming $\parametric v$ to satisfy a minor amount of additional global regularity. 
For every $\parametric v \in W^{1,p}(\parametric\Omega)$,
we find by Theorem~\ref{prop:veeser} that 
\begin{align*}
 \left| \parametric v - \Pi \parametric v \right|_{W^{1,p}(\parametric\Omega)}
 &
 \leq 
 (1+C_{\Pi}) 
 \sum_{ T \in \calT }
 \left| \parametric v - P_{T,r,1,p} \parametric v \right|_{W^{1,p}(T)}
 \\&
 \leq 
 (1+C_{\Pi}) 
 C^{\rm I}_{d,1,r,p,\mu}
 \sum_{ T \in \calT }
 \inf_{ w_{T} \in \calP_{r}(T) } 
 \left| \parametric v - w_{T} \right|_{W^{1,p}(T)}
 .
\end{align*}
The last inequality uses the best approximation property of the local interpolator. 
We now have obtained a converse to Inequality~\eqref{math:confvsdc}: 
for every $\parametric v \in W^{1,p}(\parametric\Omega)$
we have 
\begin{align}
 \label{math:dcvsconf}
 \begin{split}
  &
  \inf_{ \parametric w_{h} \in \calP_{r, \rm D}(\calT) } 
  \left| \parametric v - \parametric w_{h} \right|_{W^{1,p}(\parametric\Omega)}
  \\&\qquad\qquad
  \leq 
  (1+C_{\Pi})
  C^{\rm I}_{d,1,r,p,\mu}
  \inf_{ \parametric w_{h} \in \calP_{r,-1}(\calT) } 
  \left| \parametric v - \parametric w_{h} \right|_{W^{1,p}(\parametric\Omega)}
  .
 \end{split}
\end{align}

\begin{remark}
 The above error estimate seems to have appeared first in \cite{veeser2016approximating} 
 in the special case of the $H^{1}$-seminorm. 
 A generalization to general integer Sobolev regularity 
 and Lebesgue exponents was published later in \cite{camacho2014L2}. 
 The intended use in the later publication was similar to ours,
 namely to facilitate error estimates in the presence of variational crimes
 in surface finite element methods. 
 
 Notably, the broken Bramble-Hilbert lemma seems to close gaps 
 in earlier proofs of error estimates for isoparametric finite element methods. 
 For example, 
 the last inequality on p.1224 of \cite{bernardi1989optimal}
 is a Bramble-Hilbert-type estimate of a function on a triangle patch 
 with first-order Sobolev regularity but piecewise higher regularity;
 the proof in that reference requires a piecewise Bramble-Hilbert lemma which 
 was not in circulation at the point of publication. 
\end{remark}

\section{A Priori Error Estimates}
\label{sec:apriori}

In this section we attend to a thorough discussion of a priori error estimates for a finite element method 
of the parametric model problem and the critical role of the finite element approximation result in Theorem~\ref{prop:veeser}
in facilitating optimal convergence rates. 
We continue the discussion of a Galerkin method as in Section~\ref{sec:theory} 
with the concrete example of $\parametric V_{h} = \calP_{r,\rm D}(\calT)$ as a Galerkin space. 
\\

We consider a parametric right-hand side of the form 
\begin{align}
 \label{math:parametric:rhs}
 \parametric F( \parametric v )
 &= 
 \int_{\parametric\Omega} 
 \parametric f
 \parametric v
 \dif \parametric x
 +
 \int_{\parametric\Omega} 
 \parametric \bfg
 \cdot 
 \nabla \parametric v 
 \dif \parametric x,
 \quad 
 \parametric v \in W^{1,2}_{\rm D}(\parametric\Omega)
 . 
\end{align}
We let $\parametric u$ be the unique solution of 
\begin{gather*}
 \int_{\parametric\Omega}
 \nabla \parametric u
 \cdot 
 \parametric A
 \nabla \parametric v
 \dif \parametric x
 =
 \parametric F( \parametric v ),
 \quad 
 \parametric v \in W^{1,2}_{\rm D}(\parametric \Omega)
 .
\end{gather*}
We also let $\parametric u_{h}$ be the solution of the conforming Galerkin problem 
\begin{gather*}
 \int_{\parametric\Omega}
 \nabla \parametric u_{h}
 \cdot 
 \parametric A
 \nabla \parametric v_{h}
 \dif \parametric x
 =
 \parametric F( \parametric v_{h} ),
 \quad 
 \parametric v_{h} \in \calP_{r,\rm D}(\calT)
 .
\end{gather*}
In practice, the parametric coefficient and right-hand side can only be approximated.
Suppose that $\parametric A_{h} \in L^{\infty}(\parametric\Omega)^{n\times n}$ 
is an approximate parametric coefficient
and that we have an approximate right-hand side
\begin{align}
 \label{math:parametric:rhs:approx}
 \parametric F_{h}( \parametric v_{h} )
 &= 
 \int_{\parametric\Omega} 
 \parametric f_{h}
 \parametric v_{h}
 \dif \parametric x
 +
 \int_{\parametric\Omega} 
 \parametric \bfg_{h}
 \cdot 
 \nabla \parametric v_{h} 
 \dif \parametric x,
 \quad 
 \parametric v_{h} \in \parametric V_{h}
 ,
\end{align}
of a form analogous to the original right-hand side. 
We assume that there exists $\parametric c_{P,h} > 0$
satisfying the discrete coercivity estimate \eqref{math:coercivity:discrete}; see Remark~\ref{remark:discretecoercivity} below. 
Practically, we compute the solution $\underline{\parametric u}_{h}$
of the approximate Galerkin problem 
\begin{gather*}
 \int_{\parametric\Omega}
 \nabla \underline{\parametric u}_{h}
 \cdot 
 \parametric A_{h}
 \nabla \parametric v_{h}
 \dif \parametric x
 =
 \parametric F_{h}( \parametric v_{h} ),
 \quad 
 \parametric v_{h} \in \calP_{r,\rm D}(\calT)
 .
\end{gather*}
Our goal is estimate the error $\parametric u - \underline{\parametric u}_{h}$ in different norms.
We now recall some standard results in the Galerkin theory of elliptic problems. 
\\

The conforming Galerkin approximation is the optimal approximation in $\parametric V_{h}$
with respect to the $\parametric A$-weighted norms of the gradient:
\begin{gather}
 \label{math:apriori:galerkinseminorm}
 \| \nabla \parametric u - \nabla \parametric u_{h} \|_{L^{2}(\parametric\Omega,\parametric A)}
 = 
 \inf_{ \parametric v_{h} \in \parametric V_{h} }
 \| \nabla \parametric u - \nabla \parametric v_{h} \|_{L^{2}(\parametric\Omega,\parametric A)}
 .
\end{gather}
C\'ea's Lemma estimates the approximation in the full norm of $W^{1,2}(\parametric\Omega)$: 
\begin{gather}
 \label{math:apriori:cea}
 \sqrt{\parametric c_{P}}
 \| \parametric u - \parametric u_{h} \|_{W^{1,2}(\parametric\Omega)}
 \leq 
 \inf_{ \parametric v_{h} \in V_{h} }
 \| \nabla \parametric u - \nabla \parametric v_{h} \|_{L^{2}(\parametric\Omega,\parametric A)}
 .
\end{gather}
The error of the conforming Galerkin method in the $L^{2}$ norm is estimated by an Aubin-Nitsche-type argument,
which requires the theoretical discussion of an auxiliary problem. 
We let $\parametric z \in W^{1,2}_{\rm D}(\parametric\Omega)$ be the unique solution of 
\begin{gather}
 \label{math:dualproblem}
 \parametric B( \parametric v , \parametric z )
 = 
 \langle \parametric u - \parametric u_{h}, \parametric v \rangle_{L^{2}(\parametric\Omega)}, 
 \quad 
 \parametric v \in W^{1,2}_{\rm D}(\parametric\Omega). 
\end{gather}
We then find for $\parametric z_{h} \in \parametric V_{h}$ arbitrary that 
\begin{align*}
 \| \parametric u - \parametric u_{h} \|_{L^{2}(\parametric\Omega)}^{2}
 &= 
 \langle \parametric u - \parametric u_{h}, \parametric u - \parametric u_{h} \rangle_{L^{2}(\parametric \Omega)}
 =  
 \parametric B( \parametric u - \parametric u_{h}, \parametric z )
 =  
 \parametric B( \parametric u - \parametric u_{h}, \parametric z - \parametric z_{h} )
 ,
\end{align*}
and consequently 
\begin{gather}
 \label{math:aubinnitsche}
 \| \parametric u - \parametric u_{h} \|_{L^{2}(\parametric\Omega)}^{2}
 \leq   
 \| \nabla \parametric u - \nabla \parametric u_{h} \|_{L^{2}(\parametric \Omega,A)} 
 \| \nabla \parametric z - \nabla \parametric z_{h} \|_{L^{2}(\parametric \Omega,A)}
 . 
\end{gather}
Hence we expect the $L^{2}$ error of the conforming Galerkin method to converge generally faster 
than the error in the $W^{1,2}$ norm,
an intuition made rigorous whenever estimates for $\nabla \parametric z - \nabla \parametric z_{h}$ are available. 
\\

The corresponding error estimates for the non-conforming Galerkin approximation
reduce to estimates for $\parametric u_{h} - {\underline {\parametric u}_{h}}$, 
that is, we compare the conforming with the non-conforming approximation.
The triangle inequality in conjunction with \eqref{math:coercivity:parametric} gives 
\begin{align}
 \label{math:galerkincomparison:h1}
 \| \nabla \parametric u - \nabla {\underline {\parametric u}_{h}} \|_{L^{2}(\parametric\Omega)}
 &\leq  
 \| \nabla \parametric u - \nabla \parametric u_{h} \|_{L^{2}(\parametric\Omega)}
 + 
 \parametric c_{P}^{-\onehalf} 
 \| \nabla \parametric u_{h} - \nabla {\underline {\parametric u}_{h}} \|_{L^{2}(\parametric\Omega,\parametric A)}
 ,
 \\
 \label{math:galerkincomparison:l2}
 \| \parametric u - {\underline {\parametric u}_{h}} \|_{L^{2}(\parametric\Omega)}
 &\leq  
 \| \parametric u - \parametric u_{h} \|_{L^{2}(\parametric\Omega)}
 + 
 \parametric c_{P}^{-\onehalf} 
 \| \nabla \parametric u_{h} - \nabla {\underline {\parametric u}_{h}} \|_{L^{2}(\parametric\Omega,\parametric A)}
 . 
\end{align}
We use definitions to get 
\begin{align*}
 &
 \| \nabla \parametric u_{h} - \nabla {\underline {\parametric u}_{h}} \|_{L^{2}(\parametric\Omega,\parametric A)}^{2}
 \\&=
 \parametric B( \parametric u_{h} - {\underline {\parametric u}_{h}}, \parametric u_{h} - {\underline {\parametric u}_{h}} )
 \\&=
 \parametric F( \parametric u_{h} - {\underline {\parametric u}_{h}} )
 -
 \parametric B( {\underline {\parametric u}_{h}}, \parametric u_{h} - {\underline {\parametric u}_{h}} )
 \\&=
 \parametric F( \parametric u_{h} - {\underline {\parametric u}_{h}} )
 -
 \parametric F_{h}( \parametric u_{h} - {\underline {\parametric u}_{h}} )
 +
 \parametric B_{h}( {\underline {\parametric u}_{h}}, \parametric u_{h} - {\underline {\parametric u}_{h}} )
 -
 \parametric B( {\underline {\parametric u}_{h}}, \parametric u_{h} - {\underline {\parametric u}_{h}} )
 . 
\end{align*}
Consequently, we derive 
\begin{align}
 \label{math:testgalerkincomparison}
 \begin{split}
  &
  \| \parametric u_{h} - {\underline {\parametric u}_{h}} \|_{W^{1,2}(\parametric\Omega,\parametric A)}
  \\&\qquad\qquad
  \leq 
  \sup_{ \parametric w_{h} \in \calP_{r,\rm D}(\calT) } 
  \dfrac{ 
   ( \parametric F - \parametric F_{h} )( \parametric w_{h} )
   + 
   \parametric B_{h}( {\underline {\parametric u}_{h}}, \parametric w_{h} ) 
   -
   \parametric B( {\underline {\parametric u}_{h}}, \parametric w_{h} ) 
  }{ 
   \| \parametric w_{h} \|_{W^{1,2}(\parametric\Omega,\parametric A)} 
  }
  .
 \end{split}
\end{align}
Under natural regularity assumptions on the coefficients, 
\eqref{math:testgalerkincomparison} leads to optimal error estimates. 
In particular, in the special case that 
the conforming and non-conforming methods coincide,
\eqref{math:apriori:cea} is recovered.
\\

We need further tools to derive actual convergence rates from these abstract estimates. 
We recall a general polynomial approximation estimate
that follows from \cite[Theorem~3.1, Proposition~6.1]{dupont1980polynomial}
and a scaling argument. 

\begin{lemma}
 \label{prop:denylions}
 Let $T \in \calT$ be a $d$-simplex, $p \in [1,\infty]$, and $r \in \bbN_{0}$.
 Let $s \in \bbR$ with $r < s \leq r+1$.
 Then there exists $C^{\rm BH}_{d,\mu,s} > 0$
 that depends only on $d$, $s$, and $\mu(\calT)$ such that 
 for all $m \in \bbN_{0}$ with $m \leq r$ we get 
 \begin{gather}
  \label{math:bramblehilbert}
  \inf_{ \parametric v_{h} \in \calP_{r}(T) } 
  \left| \parametric v - \parametric v_{h} \right|_{W^{m,p}(T)}
  \leq 
  C^{\rm BH}_{d,\mu,s}
  h_{T}^{ s-m }
  \left| \parametric v \right|_{W^{s,p}(T)},
  \quad 
  \parametric v \in W^{s,p}(T)
  . 
 \end{gather}
\end{lemma}

We use the polynomial approximation result to derive quantitative error estimates 
in terms of the mesh size. 
We write ${\underline {\physical u}_{h}} := \Phi^{-\ast} {\underline {\parametric u}_{h}}$
for the transformation of the non-conforming Galerkin solution onto the physical domain. 
It is then easily verified that 
\begin{align*}
 \| \physical u - {\underline {\physical u}_{h}} \|_{L^{2}(\physical\Omega}
 &\leq 
 \| \det \Dif \Phi^{-1} \|_{L^{\infty}(\parametric\Omega)}
 \| \parametric u - {\underline {\parametric u}_{h}} \|_{L^{2}(\parametric\Omega)}
 ,
 \\ 
 \| \nabla \physical u - \nabla {\underline {\physical u}_{h}} \|_{L^{2}(\physical\Omega}
 &\leq 
 \| \det \Dif \Phi^{-1} \|_{L^{\infty}(\parametric\Omega)}
 \| \Dif \Phi^{-1} \|_{L^{\infty}(\parametric\Omega)}
 \| \nabla \parametric u - \nabla {\underline {\parametric u}_{h}} \|_{L^{2}(\parametric\Omega)}
 .
\end{align*}
Due to \eqref{math:galerkincomparison:h1} and \eqref{math:galerkincomparison:l2},
it remains to analyze the error $\parametric u - \parametric u_{h}$
of the conforming Galerkin method (over the parametric domain),
and the non-conformity error terms in \eqref{math:testgalerkincomparison}.
\\

The physical setting captures the relevant regularity features of the problem.
Local regularity features of the physical solution translate to local regularity features 
of the parametric solution because the transformation $\Phi$ is piecewise smooth by assumption. 
Given $s \in \bbR^{+}_{0}$, there exists $C^{\Phi}_{n,s} > 0$,
bounded in terms of $n$, $s$, 
$\|\Phi\|_{W^{l+1,\infty}(\parametric\Omega)}$ 
and $\|\Phi^{-1}\|_{W^{l+1,\infty}(\physical\Omega)}$,
such that whenever we have $\physical u \in W^{s,2}(\physical\Omega)$, 
we get on every $n$-simplex $T \in \calT$
\begin{gather*}
 | \parametric u |_{W^{s,2}(T)}
 \leq 
 C^{\Phi}_{n,s}
 \| \physical u \|_{W^{s,2}(\Phi(T))}. 
\end{gather*}
So suppose that $\physical u \in W^{s,2}(\Phi(T))$ for some $s \geq 1$, 
and let $r \in \bbN$ the largest integer with $r < s$. 
In conjunction with the Galerkin optimality, Theorem~\ref{prop:veeser}, and Lemma~\ref{prop:denylions},
we then obtain the estimate 
\begin{align*}
 | \parametric u - \parametric u_{h} |_{W^{1,2}(\parametric\Omega)} 
 &=  
 \inf_{ \parametric v_{h} \in \calP_{r,\rm D}(\calT) }
 | \parametric u - \parametric v_{h} |_{W^{1,2}(\parametric\Omega)} 
 \\&=
 (1+C_{\Pi})
 \sum_{ T \in \calT }
 \left| \parametric v - P_{T,r,1} \parametric v \right|_{W^{1,2}(T)}
 \\&\leq 
 (1+C_{\Pi})
 C^{\rm I}_{d,1,r,2,\mu}
 C^{\rm BH}_{d,\mu,s} 
 \sum_{ T \in \calT }
 h_{T}^{s-1}
 | \parametric u |_{W^{s,2}(T)}
 .
\end{align*}
Recall the definition of $\parametric z$
as the solution of the parametric model problem with right-hand side $\parametric u - \parametric u_{h}$.
It then follows that $\physical z := \Phi^{-\ast} \parametric z$
is the solution of the physical model problem with right-hand side 
$\Phi^{-\ast}( \parametric u - \parametric u_{h} )$. 
Consequently, 
whenever $\physical z \in W^{t,2}(\physical\Omega)$
for some $t \in \bbR^{+}_{0}$,
then for all $n$-simplices $T \in \calT$
it follows that 
$| \parametric u |_{W^{t,2}(T)} \leq  C^{\Phi}_{n,t} \| \physical u \|_{W^{t,2}(\Phi(T))}$. 
Similar to the arguments used above, we then see 
\begin{align}
 | \parametric z - \parametric z_{h} |_{W^{1,2}(\parametric\Omega)} 
 \leq 
 (1+C_{\Pi})
 C^{\rm I}_{d,1,r,\mu}
 C^{\rm BH}_{d,\mu,t} 
 \sum_{ T \in \calT }
 h_{T}^{t-1}
 | \parametric z |_{W^{t,2}(T)}
 . 
\end{align}
This provides convergence rates for the conforming Galerkin method. 
\\

For error estimates for the non-conforming problem, 
it remains to analyze the difference 
$\nabla \parametric u_{h} - \nabla {\underline {\parametric u}_{h}}$
using \eqref{math:testgalerkincomparison}.
Via \eqref{math:coercivity:parametric} and \eqref{math:discretestability:nonconforming}
we get 
\begin{align*}
 \begin{split}
  &
  \| \parametric u_{h} - {\underline {\parametric u}_{h}} \|_{W^{1,2}(\parametric\Omega,\parametric A)}
  \\&
  \qquad
  \leq 
  \frac{\parametric c_{P,h}}{\parametric c_{P}}
  \left( 
   \| \parametric F - \parametric F_{h} \|_{W^{-1,2}_{\rm N}(\parametric\Omega)}
   + 
   \| \parametric A - \parametric A_{h} \|_{L^{\infty}(\parametric\Omega)}
   | {\underline {\parametric u}_{h}} |_{W^{1,2}(\parametric\Omega)}
  \right)
  .
 \end{split}
\end{align*}
Note that
\begin{align*}
 \parametric c_{P,h}
 | {\underline {\parametric u}_{h}} |_{W^{1,2}(\parametric\Omega)}
 \leq 
 \| \parametric F_{h} \|_{W^{-1,2}_{\rm N}(\parametric\Omega)}
 \leq 
 \| \parametric f_{h} \|_{L^{2}(\parametric\Omega)}
 +
 \| \parametric \bfg_{h} \|_{L^{2}(\parametric\Omega)}
\end{align*}
As for the approximate right-hand side $F_{h}$ as in \eqref{math:parametric:rhs:approx}, 
we get 
\begin{align*}
 \|
  \parametric F_{ } 
  -
  \parametric F_{h} 
 \|_{W^{-1,2}_{\rm N}(\parametric\Omega)}
 &\leq 
 \| \parametric f - \parametric f_{h} \|_{L^{2}(\parametric\Omega)}
 +
 \| \parametric \bfg - \parametric \bfg_{h} \|_{L^{2}(\parametric\Omega)}
 .
\end{align*}
Consequently, it remains to bound the errors of the approximate right-hand side 
and the approximate coefficient. 
Let us make the regularity assumptions
\begin{gather*}
 \physical A \in W^{l,\infty}(\physical\Omega),
 \quad 
 \physical f \in W^{l,2}(\physical\Omega),
 \quad 
 \physical \bfg \in W^{l,2}(\physical\Omega).
\end{gather*}
For each $n$-simplex $T \in \calT$ it then follows that 
\begin{align}
 | \parametric f |_{W^{l,2}(T)}
 &\leq 
 C^{\Phi}_{n,l} 
 \| \physical f \|_{W^{l,2}(\Phi(T))}
 ,
 \\ 
 | \parametric \bfg |_{W^{l,2}(T)}
 &\leq   
 C^{\Phi}_{n,l} 
 \| \physical \bfg \|_{W^{l,2}(\Phi(T))}
 ,
 \\
 | \parametric A |_{W^{l,\infty}(T)}
 &\leq   
 C^{\Phi}_{n,l} 
 \| \physical A \|_{W^{l,\infty}(\Phi(T))}^{}
 .
\end{align}
In other words, the parametric data and coefficients inherit 
the piecewise regularity of their physical counterparts. 
Let $k \in \bbN_{0}$ be the largest integer with $k < l$; 
we use $k$ as the polynomial degree of the data approximation. 
Suppose that we specifically choose the approximate parametric coefficient $\parametric A_{h}$ 
in each component as the piecewise best $L^{2}$ approximation of $\parametric A$ 
by polynomials of degree at most $k$.
Similarly, suppose we choose $\parametric f_{h}$ and $\parametric \bfg_{h}$ 
as piecewise polynomial best $L^{2}$ approximations of degree at most $k$. 
Then Lemma~\ref{prop:denylions} yields 
\begin{align*}
 \| \parametric A - \parametric A_{h} \|_{L^{\infty}(T)}
 &\leq 
 C^{\rm BH}_{d,\mu,l}
 h_{T}^{l} 
 | \parametric A |_{W^{l,\infty}(T)}
 ,
 \\
 \| \parametric f - \parametric f_{h} \|_{L^{2}(T)}
 &\leq 
 C^{\rm BH}_{d,\mu,l}
 h_{T}^{l} 
 | \parametric f |_{W^{l,2}(T)}
 ,
 \\ 
 \| \parametric \bfg - \parametric \bfg_{h} \|_{L^{2}(T)}
 &\leq   
 C^{\rm BH}_{d,\mu,l}
 h_{T}^{l} 
 | \parametric \bfg |_{W^{l,2}(T)}
\end{align*}
for each $n$-simplex $T \in \calT$. 
This completes the desired estimates.

\begin{remark}
 We compare our approach to a classical finite element method for curved domains, namely isoparametric finite element methods.
 The latter assume an affine mesh $\calT_{h}$ of a polyhedral parametric domain $\parametric\Omega_{h}$,
 and a piecewise polynomial coordinate transformation $\Phi_{h}$
 whose image is a curved polyhedral domain that ``approximates'' (in whatever sense) the physical domain. 
 The finite element method over the image of $\Phi$ can be pulled back to a finite element method 
 over the polyhedral parametric domain $\parametric\Omega_{h}$. 
 Since each $\Phi_{h}$ is piecewise polynomial,
 the pullback introduces only polynomial terms that can be evaluated exactly. 
 We may thus interpret the approximate isoparametric meshes 
 as a tool to develop approximate coefficients on the underlying affine mesh. 
 
 Isoparametric finite element methods use an approximate coordinate transformation 
 whose coefficients can be determined exactly.
 Our approach in this article is strictly different 
 in that we use an exact coordinate transformation whose coefficients are then approximated. 
 It seems that both procedures lead to similar results in practice. 
 
 One of the first error estimates for the isoparametric finite element method
 is due to Ciarlet and Raviart \cite{ciarlet1972interpolation}, 
 who transferred the canonical interpolant to the isoparametric setting. 
 The idea is to take the nodal values of the canonical interpolant within the physical domain. 
 This idea can be replicated in our setting without essential difficulty. 
 However, a severe restriction is the requirement of higher regularity for the solution. 
 A more refined error analysis for the case of conforming geometry 
 was established by Clement \cite{clement1975approximation}, Scott and Zhang \cite{scott1990finite},
 and recently by Ern and Guermond \cite{ern2017finite}
 Lenoir's contribution \cite{lenoir1986optimal} focuses on the construction of curved triangulations 
 for a given domain and gives error estimates again only for the case that 
 the physical solution is continuous (see Lemma~7 in the reference).
 
 To our best knowledge, however, 
 it is only with Inequality~\eqref{math:dcvsconf} that rigorous a priori error estimates 
 for higher order isoparametric FEM are available.
 Furthermore, our method is more flexible whenever the (possibly non-polynomial)
 coordinate transformation is actually known. 
\end{remark}

\begin{remark}
 \label{remark:discretecoercivity}
 The discrete coercivity condition \eqref{math:coercivity:discrete} 
 holds for sufficiently fine approximation of the parametric coefficient. 
 The feasibility of quasi-optimal positivity preserving interpolation is not studied in this article;
 we refer to the literature for results on positivity preserving interpolation of functions
 \cite{nochetto2002positivity}. 
\end{remark}

\begin{remark}
 Let $r \in \bbN$ and $k \in \bbN_{0}$. 
 Suppose that $\physical u \in W^{r+1,2}(\physical\Omega)$
 and that $\physical A$, $\physical f$, and $\physical \bfg$
 have regularity $W^{k+1,2}(\physical\Omega)$. 
 Let us also assume that full elliptic regularity holds for the physical model problem,
 so that $\physical z \in W^{2,2}(\physical\Omega)$.
 Letting $C > 0$ denote a generic constant
 and letting $h > 0$ denote the maximum diameter of a simplex in the triangulation,
 we have 
 \begin{gather*}
  | \physical u - \underline{\physical u}_{h} |_{W^{1,2}(\physical\Omega)}
  \leq 
  C
  h^{r}
  \| \physical u \|_{W^{r+1,2}(\physical\Omega)}
  + 
  C
  h^{k+1}
  \| \physical u \|_{W^{1,2}(\physical\Omega)},
  \\
  \| \physical u - \underline{\physical u}_{h} \|_{L^{2}(\physical\Omega)}
  \leq 
  C
  h^{r+1}
  \| \physical u \|_{W^{r+1,2}(\physical\Omega)}
  + 
  C
  h^{k+1}
  \| \physical u \|_{W^{1,2}(\physical\Omega)}
  . 
 \end{gather*}
 Analogous estimates are known for surface finite element methods,
 where $k$ signifies the degree of geometric approximation. 
 Our a priori error estimate is the sum of a classical Galerkin approximation error
 and additional error terms due to coefficient and data approximation. 
 This is reminiscent 
 of what has been called ``almost best-approximation error'' 
 and ``geometric error'' in the analysis of surface finite element methods 
 \cite[p.2]{demlow2009higher}. 
\end{remark}

\section{Examples and Computational Experiments}
\label{sec:examples}

This articles finishes with a few examples of coordinate transformations 
from polyhedral parametric domains onto curved physical domains 
and illustrative numerical computations within those geometric setups.

We have conducted our numerical experiments with the \texttt{FEniCS} Library. 
We list the observed errors for polynomial order $1 \leq r \leq 4$
and several levels of uniform refinement. 
The polynomial degree of the interpolated coefficients and data has always been 
equal to
the degree of the finite element space. 
All linear system of equations have been solved with the conjugate gradient method
and \texttt{FEniCS}'s built-in \texttt{amg} preconditioner,
the absolute and relative error tolerance set to \texttt{1e-20} each.
The theoretically predicted convergence rates are achieved in practice.

\subsection{Anulus}
Models in geophysics and climate science 
assume static homogeneous conditions over a large interior part of the Earth
and pose partial differential equations only over a thin outer volume of the planet.
The equations reign over an $n$-dimensional anulus,
that is, an $n$-ball with an internal $n$-ball removed. 

We introduce the parametric anulus $\parametric \calA$
and the physical anulus $\physical \calA$ by 
\begin{gather*}
 \parametric \calA := \left\{ \parametric x \in \bbR^{n} \;\middle|\; \frac{1}{2} < \|\parametric x\|_{1} < 1 \right\},
 \quad
 \physical \calA := \left\{ \physical x \in \bbR^{n} \;\middle|\; \frac{1}{2} < \|\physical x\|_{2} < 1 \right\}.
\end{gather*}
We consider a coordinate transformations
from the parametric anulus onto the physical anulus, 
\begin{gather*} 
 \Phi : \parametric \calA \rightarrow \physical \calA,
 \quad 
 \parametric x
 = 
 \|\parametric x\|_{1} \|\parametric x\|_{2}^{-1}
 \parametric x,
\end{gather*}
which is easily seen to be invertible and bi-Lipschitz. 
Furthermore, 
over the intersections of $\parametric \calA$ 
with any of the $2^{n}$ coordinate quadrants, 
the transformation $\Phi$ is a diffeomorphism with derivatives of all orders pointwise bounded.
In particular, 
it is easy to find a (coarse) initial triangulation of $\parametric\calA$ 
such that the coordinate transformation $\Phi$ is piecewise smooth. 

This construction allows us to transport partial differential equations 
over an \emph{Euclidean anulus} $\physical \calA$
to partial differential equations over the polyhedral \emph{Manhattan-metric anulus} $\parametric \calA$.
Suppose that we want to solve the Poisson problem 
\begin{gather}
 \label{math:physical_example_problem}
 \int_{\physical \calA} \nabla \physical u \cdot \nabla \physical v \dif \physical x
 = 
 \int_{\physical \calA} \physical f \cdot \physical v \dif \physical x 
 + 
 \int_{\physical \calA} \physical \bfg \cdot \nabla \physical v \dif \physical x,
 \quad 
 \physical v \in H^{1}_{0}(\physical \calA),
\end{gather}
over the Euclidean anulus. Along the transformation $\Phi$
we can translate this into an equivalent Poisson problem 
over the parametric anulus $\parametric \calA$ of the form 
\begin{align}
 \label{math:parametric_example_problem}
 \begin{split}
  &
  \int_{\parametric \calA} 
  \left| \det \Dif \Phi \right| 
  \nabla \parametric u \cdot \Dif\Phi^{-1}_{|\Phi}\Dif\Phi^{-t}_{|\Phi} \nabla \parametric v \dif \parametric x
  \\&\quad= 
  \int_{\parametric \calA} 
  \left| \det \Dif \Phi \right| 
  (\physical f \circ \Phi) \parametric v \dif \parametric x 
  + 
  \int_{\parametric\Omega} 
  \left| \det \Dif \Phi \right| 
  \left( \Dif\Phi^{-1}_{|\Phi} \left( \physical\bfg \circ \Phi \right) \right) \nabla \parametric v \dif \parametric x
 \end{split}
\end{align}
for any test function $\parametric v \in H^{1}_{0}(\parametric \calA)$.
Thus the parametric problem can be solved with textbook methods. 
However, it is a stronger result and a consequence of Theorem~\ref{prop:veeser} 
that the piecewise regularity of the parametric solution $\parametric u$,
which is inherited from the physical solution $\physical u$ on each cell,
leads to the same convergence rates that the corresponding global regularity of $\physical u$ would suggest. 

\begin{figure}[t]
 \centering 
 \begin{tikzpicture}
    [line join=bevel,x={( 1.5cm, 0mm)},y={( 0mm, 1.5cm)},z={( 1.5*3.85mm, -1.5*3.85mm)}]
    
    \coordinate (O)   at (  0.0,  0.0, 0.0);
    
    \filldraw[fill=black!15!white, draw=black] (O) circle (1.5cm);
    \filldraw[fill=white, draw=black] (O) circle (0.75cm);

  \end{tikzpicture}  \begin{tikzpicture}
    [line join=bevel,x={( 1.5cm, 0mm)},y={( 0mm, 1.5cm)},z={( 1.5*3.85mm, -1.5*3.85mm)}]
    
    \coordinate (O)   at (  0.0,  0.0, 0.0);
    
    \coordinate (EEo) at (  1.0,  0.0, 0.0);
    \coordinate (NNo) at (  0.0,  1.0, 0.0);
    \coordinate (WWo) at ( -1.0,  0.0, 0.0);
    \coordinate (SSo) at (  0.0, -1.0, 0.0);
    
    \coordinate (EEi) at ($(O)!0.5!(EEo)$);
    \coordinate (NNi) at ($(O)!0.5!(NNo)$);
    \coordinate (WWi) at ($(O)!0.5!(WWo)$);
    \coordinate (SSi) at ($(O)!0.5!(SSo)$);
    
    \coordinate (NEo) at ($(NNo)!0.5!(EEo)$);
    \coordinate (NWo) at ($(NNo)!0.5!(WWo)$);
    \coordinate (SWo) at ($(SSo)!0.5!(WWo)$);
    \coordinate (SEo) at ($(SSo)!0.5!(EEo)$);
    
    \coordinate (NEi) at ($(O)!0.5!(NEo)$);
    \coordinate (NWi) at ($(O)!0.5!(NWo)$);
    \coordinate (SWi) at ($(O)!0.5!(SWo)$);
    \coordinate (SEi) at ($(O)!0.5!(SEo)$);
    
    \filldraw[fill=black!15!white, draw=black] (EEo) -- (EEi) -- (NNi) -- cycle;
    \filldraw[fill=black!15!white, draw=black] (EEo) -- (NNi) -- (NNo) -- cycle;
    
    \filldraw[fill=black!15!white, draw=black] (NNo) -- (NNi) -- (WWi) -- cycle;
    \filldraw[fill=black!15!white, draw=black] (NNo) -- (WWi) -- (WWo) -- cycle;
    
    \filldraw[fill=black!15!white, draw=black] (WWo) -- (WWi) -- (SSi) -- cycle;
    \filldraw[fill=black!15!white, draw=black] (WWo) -- (SSi) -- (SSo) -- cycle;
    
    \filldraw[fill=black!15!white, draw=black] (SSo) -- (SSi) -- (EEi) -- cycle;
    \filldraw[fill=black!15!white, draw=black] (SSo) -- (EEi) -- (EEo) -- cycle;

  \end{tikzpicture}  \begin{tikzpicture}
    [line join=bevel,x={( 1.5cm, 0mm)},y={( 0mm, 1.5cm)},z={( 1.5*3.85mm, -1.5*3.85mm)}]
    
    \coordinate (O)   at (  0.0,  0.0, 0.0);
    
    \coordinate (EEo) at (  1.0,  0.0, 0.0);
    \coordinate (NNo) at (  0.0,  1.0, 0.0);
    \coordinate (WWo) at ( -1.0,  0.0, 0.0);
    \coordinate (SSo) at (  0.0, -1.0, 0.0);
    
    \coordinate (EEi) at ($(O)!0.5!(EEo)$);
    \coordinate (NNi) at ($(O)!0.5!(NNo)$);
    \coordinate (WWi) at ($(O)!0.5!(WWo)$);
    \coordinate (SSi) at ($(O)!0.5!(SSo)$);
    
    \coordinate (NEo) at ($(NNo)!0.5!(EEo)$);
    \coordinate (NWo) at ($(NNo)!0.5!(WWo)$);
    \coordinate (SWo) at ($(SSo)!0.5!(WWo)$);
    \coordinate (SEo) at ($(SSo)!0.5!(EEo)$);
    
    \coordinate (NEi) at ($(O)!0.5!(NEo)$);
    \coordinate (NWi) at ($(O)!0.5!(NWo)$);
    \coordinate (SWi) at ($(O)!0.5!(SWo)$);
    \coordinate (SEi) at ($(O)!0.5!(SEo)$);
    
    \filldraw[fill=black!15!white, draw=black] (EEo) -- (EEi) -- (NEo) -- cycle;
    \filldraw[fill=black!15!white, draw=black] (EEi) -- (NNi) -- (NEo) -- cycle;
    \filldraw[fill=black!15!white, draw=black] (NNo) -- (NNi) -- (NEo) -- cycle;
    
    \filldraw[fill=black!15!white, draw=black] (WWo) -- (WWi) -- (NWo) -- cycle;
    \filldraw[fill=black!15!white, draw=black] (WWi) -- (NNi) -- (NWo) -- cycle;
    \filldraw[fill=black!15!white, draw=black] (NNo) -- (NNi) -- (NWo) -- cycle;
    
    \filldraw[fill=black!15!white, draw=black] (WWo) -- (WWi) -- (SWo) -- cycle;
    \filldraw[fill=black!15!white, draw=black] (SSi) -- (WWi) -- (SWo) -- cycle;
    \filldraw[fill=black!15!white, draw=black] (SSo) -- (SSi) -- (SWo) -- cycle;
    
    \filldraw[fill=black!15!white, draw=black] (EEo) -- (EEi) -- (SEo) -- cycle;
    \filldraw[fill=black!15!white, draw=black] (EEi) -- (SSi) -- (SEo) -- cycle;
    \filldraw[fill=black!15!white, draw=black] (SSo) -- (SSi) -- (SEo) -- cycle;

  \end{tikzpicture}  \caption{
 From left to right: 
 physical anulus $\physical \calA$, 
 and parametric anulus $\parametric \calA$ with two possible triangulations 
 for which $\Phi$ is a piecewise diffeomorphism. 
 Our computations use the first triangulation. 
 }
 \label{fig:triangulation:A}
\end{figure}

\begin{remark}
 For illustration, consider the radially symmetric Dirichlet problem 
 \begin{align}
  - \Delta \physical u = 1, \quad \physical u_{|\partial\physical \calA} = 0, 
 \end{align}
 over the physical anulus $\physical A$ with $n=2$. 
 The solution is the function 
 \begin{align}
  \physical u(x,y) 
  = 
  \frac{1}{4}
  +
  \frac{3\ln(x^{2}+y^{2})}{32 \ln(2)}
  -
  \frac{x^{2}+y^{2}}{4}
  .
 \end{align}
 The results of the computational experiments (for the corresponding parametric) problems 
 are summarized in Table~\ref{table:anulus_problem_1}.
 We have used the first triangulation in Figure~\ref{fig:triangulation:A}
 The expected convergence behavior is clearly visible for all polynomial orders. 
\end{remark}

\subsection{Quadrant of Unit Ball}
Our second example geometry considers the positive quadrant of the Euclidean unit ball. 
We transport differential equations over that domain 
to differential equations over the positive quadrant of the Manhattan unit ball.
The homeomorphism is the identity near the origin. 
We concretely define
\begin{gather*}
 \parametric \calB := \left\{ \parametric x \in (\bbR_{0}^{+})^{n} \;\middle|\; \|\parametric x\|_{1} < 1 \right\},
 \quad
 \physical \calB := \left\{ \physical x \in (\bbR_{0}^{+})^{n} \;\middle|\; \|\physical x\|_{2} < 1 \right\}.
\end{gather*}
Consider the transformation 
\begin{gather*}
 \Psi_{}^{} : \parametric \calB \rightarrow \physical \calB,
 \quad 
 \parametric x \mapsto \left\{\begin{array}{ll}
            \parametric x 
            & \text{ if } \|\parametric x\|_1 \leq \onehalf,
            \\
            \left( 
              \|\parametric x\|_{1}^{-1}
              -
              \|\parametric x\|_{2}^{-1}
              +
              2\frac{\|\parametric x\|_{1}}{\|\parametric x\|_{2}}
              -
              1
            \right)
            \parametric x
            & \text{ if } \onehalf < \|\parametric x\|_1 \leq 1.
           \end{array}
 \right.
\end{gather*}
This mapping is the identity over the set of points with Manhattan distance at most $\onehalf$ from the origin.
It is easily verified that both $\Psi$ and $\Psi^{-1}$ are bi-Lipschitz.
If a triangulation of $\parametric \calB$ accommodates the case distinction in the definition,
then both $\Psi$ and $\Psi^{-1}$ have bounded derivatives of all orders over each cell.

\begin{figure}[t]
 \centering 
 \begin{tikzpicture}
    [line join=bevel,x={( 1.5cm, 0mm)},y={( 0mm, 1.5cm)},z={( 1.5*3.85mm, -1.5*3.85mm)}]
    
    \filldraw[fill=black!15!white, draw=black]
    (0,0) -- (1.5cm,0mm) arc (0:90:1.5cm) -- (0,0);

  \end{tikzpicture}  \begin{tikzpicture}
    [line join=bevel,x={( 1.5cm, 0mm)},y={( 0mm, 1.5cm)},z={( 1.5*3.85mm, -1.5*3.85mm)}]
    
    \coordinate (P0) at (  0.0,  0.0, 0.0);
    
    \coordinate (P1) at (  0.5,  0.0, 0.0);
    \coordinate (P2) at (  0.0,  0.5, 0.0);
    \coordinate (P3) at (  1.0,  0.0, 0.0);
    \coordinate (P4) at (  0.0,  1.0, 0.0);
    
    \filldraw[fill=black!15!white, draw=black] (P0) -- (P1) -- (P2) -- cycle;
    \filldraw[fill=black!15!white, draw=black] (P1) -- (P2) -- (P3) -- cycle;
    \filldraw[fill=black!15!white, draw=black] (P2) -- (P3) -- (P4) -- cycle;
    
  \end{tikzpicture}  \begin{tikzpicture}
    [line join=bevel,x={( 1.5cm, 0mm)},y={( 0mm, 1.5cm)},z={( 1.5*3.85mm, -1.5*3.85mm)}]
    
    \coordinate (P0) at (  0.0,  0.0, 0.0);
    \coordinate (P1) at (  0.5,  0.0, 0.0);
    \coordinate (P2) at (  0.0,  0.5, 0.0);
    \coordinate (P3) at (  1.0,  0.0, 0.0);
    \coordinate (P4) at (  0.0,  1.0, 0.0);
    \coordinate (P5) at (  0.5,  0.5, 0.0);
    
    \filldraw[fill=black!15!white, draw=black] (P0) -- (P1) -- (P2) -- cycle;
    \filldraw[fill=black!15!white, draw=black] (P1) -- (P2) -- (P5) -- cycle;
    \filldraw[fill=black!15!white, draw=black] (P1) -- (P3) -- (P5) -- cycle;
    \filldraw[fill=black!15!white, draw=black] (P2) -- (P4) -- (P5) -- cycle;
    
  \end{tikzpicture}  \caption{
 From left to right: 
 physical domain $\physical \calB$, 
 and parametric domain $\parametric \calB$ with two possible triangulations 
 for which $\Psi$ is a piecewise diffeomorphism. 
 Our computations use the first triangulation. 
 }
 \label{fig:triangulation:B}
\end{figure}

\begin{remark}
 For the purpose of demonstration,
 we let $n=2$ and solve the Poisson problem over $\physical \calB$ 
 with homogeneous Dirichlet boundary conditions: 
 \begin{align}
  -\Delta \physical u
  = 
  24 x^{2} y^{2} - 2 ( x^{2} + y^{2} ) + 2( x^{4} + y^{4} ),
  \quad 
  \physical u_{|\partial\physical \calB} = 0 
  .
 \end{align}
 The solution is the polynomial 
 \begin{align}
  \physical u 
  =
  x^{2} y^{2}( 1 - x^{2} - y^{2} ) 
  .
 \end{align} 
 Though the function $\physical u$ is even a polynomial, 
 its parametric counterpart $\parametric u$ is not. 
 The results of the computational experiments (for the corresponding parametric problem) 
 are summarized in Table~\ref{table:ball_problem_1}.
 We have used the first triangulation in Figure~\ref{fig:triangulation:B}
 The theoretically predicted convergence behavior emerges.
\end{remark}

\begin{table}[t]
  \captionof{table}{Convergence table for example problem over the anulus.}
  \label{table:anulus_problem_1}
  \centering\footnotesize
  
  \subfloat[$H^{1}$ seminorm of error $e$ and convergence rate $\rho$]{
    \begin{tabular} {r r l r l r l r l}
      \toprule
      & \multicolumn{2}{c}{$r=1$}
      & \multicolumn{2}{c}{$r=2$}
      & \multicolumn{2}{c}{$r=3$}
      & \multicolumn{2}{c}{$r=4$}
      \\
      \cmidrule(lr){2-3}
      \cmidrule(lr){4-5}
      \cmidrule(lr){6-7}
      \cmidrule(lr){8-9}
      $L$ & $|e|_{1}$ & $\rho$ & $|e|_{1}$ & $\rho$ & $|e|_{1}$ & $\rho$ & $|e|_{1}$ & $\rho$
      \\[0.5em]
        0	& 0.25098	& --	& 0.022662	& --	& 0.0092381	& --	& 0.0044456	& --
        \\
        1	& 0.17023	& 0.56	& 0.011661	& 0.95	& 0.0019261	& 2.26	& 0.00038776	& 3.51
        \\
        2	& 0.12360	& 0.46	& 0.0059485	& 0.97	& 6.1275e-04	& 1.65	& 7.0433e-05	& 2.46
        \\
        3	& 0.06808	& 0.86	& 0.0016833	& 1.82	& 8.9297e-05	& 2.77	& 5.5085e-06	& 3.67
        \\
        4	& 0.035349	& 0.94	& 4.3925e-04	& 1.93	& 1.1607e-05	& 2.94	& 3.7089e-07	& 3.89
        \\
        5	& 0.017975	& 0.97	& 1.1169e-04	& 1.97	& 1.4614e-06	& 2.98	& 2.3773e-08	& 3.96
        \\
        6	& 0.0090598	& 0.98	& 2.8134e-05	& 1.98	& 1.8270e-07	& 2.99	& 1.5005e-09	& 3.98
        \\
      \bottomrule
    \end{tabular}
  }

  \subfloat[Convergence in $L^{2}$ norm and convergence rate $\rho$]{
    \begin{tabular} {r r l r l r l r l}
      \toprule
      & \multicolumn{2}{c}{$r=1$}
      & \multicolumn{2}{c}{$r=2$}
      & \multicolumn{2}{c}{$r=3$}
      & \multicolumn{2}{c}{$r=4$}
      \\
      \cmidrule(lr){2-3}
      \cmidrule(lr){4-5}
      \cmidrule(lr){6-7}
      \cmidrule(lr){8-9}
      $L$ & $\|e\|_{2}$ & $\rho$ & $\|e\|_{2}$ & $\rho$ & $\|e\|_{2}$ & $\rho$ & $\|e\|_{2}$ & $\rho$
      \\[0.5em]
        0	& 0.028212	& --	& 0.0014496	& --  	& 6.6475e-04	& --   	& 2.3259e-04	& --
        \\
        1	& 0.015748	& 0.84	& 6.4554e-04	& 1.16	& 9.1570e-05	& 2.85	& 1.5631e-05	& 3.89
        \\
        2	& 0.0089178	& 0.82	& 2.0229e-04	& 1.67	& 1.4589e-05	& 2.64	& 1.4000e-06	& 3.48
        \\
        3	& 0.0026699	& 1.73	& 2.5682e-05	& 2.97	& 9.7970e-07	& 3.89	& 5.1223e-08	& 4.77
        \\
        4	& 7.0914e-04	& 1.91	& 3.1727e-06	& 3.01	& 5.9079e-08	& 4.05	& 1.7106e-09	& 4.90
        \\
        5	& 1.8158e-04	& 1.96	& 3.9619e-07	& 3.00	& 3.5422e-09	& 4.05	& 5.4605e-11	& 4.96
        \\
        6	& 4.5877e-05	& 1.98	& 4.9728e-08	& 2.99	& 2.1552e-10	& 4.03	& 1.7219e-12	& 4.98
        \\
      \\
      \bottomrule
    \end{tabular}
  }
\end{table} 

\begin{table}[t]
  \captionof{table}{Convergence table for example problem over the positive ball quadrant.}
  \label{table:ball_problem_1}
  \centering\footnotesize
  
  \subfloat[$H^{1}$ seminorm of error $e$ and convergence rate $\rho$]{
    \begin{tabular} {r r l r l r l r l}
      \toprule
      & \multicolumn{2}{c}{$r=1$}
      & \multicolumn{2}{c}{$r=2$}
      & \multicolumn{2}{c}{$r=3$}
      & \multicolumn{2}{c}{$r=4$}
      \\
      \cmidrule(lr){2-3}
      \cmidrule(lr){4-5}
      \cmidrule(lr){6-7}
      \cmidrule(lr){8-9}
      $L$ & $|e|_{1}$ & $\rho$ & $|e|_{1}$ & $\rho$ & $|e|_{1}$ & $\rho$ & $|e|_{1}$ & $\rho$
      \\[0.5em]
        0	& 0.10633	& --	& 0.11316	& --	& 0.054669	& --	& 0.036351	& --
        \\
        1	& 0.10436	& 0.02	& 0.060710	& 0.89	& 0.020062	& 1.44	& 0.0051925	& 2.80
        \\
        2	& 0.083847	& 0.31	& 0.029172	& 1.05	& 0.0055015	& 1.86	& 6.5194e-04	& 2.99
        \\
        3	& 0.052386	& 0.67	& 0.0093838	& 1.63	& 8.6523e-04	& 2.66	& 5.1075e-05	& 3.67
        \\
        4	& 0.030721	& 0.76	& 0.0027669	& 1.76	& 1.2138e-04	& 2.83	& 3.4510e-06	& 3.88
        \\
        5	& 0.016719	& 0.87	& 7.5126e-04	& 1.88	& 1.6042e-05	& 2.91	& 2.2429e-07	& 3.94
        \\
        6	& 0.0087313	& 0.93	& 1.9570e-04	& 1.94	& 2.0609e-06	& 2.96	& 1.4286e-08	& 3.97
        \\
      \bottomrule
    \end{tabular}
  }
  
  \subfloat[Convergence in $L^{2}$ norm and convergence rate $\rho$]{
    \begin{tabular} {r r l r l r l r l}
      \toprule
      & \multicolumn{2}{c}{$r=1$}
      & \multicolumn{2}{c}{$r=2$}
      & \multicolumn{2}{c}{$r=3$}
      & \multicolumn{2}{c}{$r=4$}
      \\
      \cmidrule(lr){2-3}
      \cmidrule(lr){4-5}
      \cmidrule(lr){6-7}
      \cmidrule(lr){8-9}
      $L$ & $\|e\|_{2}$ & $\rho$ & $\|e\|_{2}$ & $\rho$ & $\|e\|_{2}$ & $\rho$ & $\|e\|_{2}$ & $\rho$
      \\[0.5em]
        0	& 0.0088039	& --	& 0.0066160	& --	& 0.0034871	& --	& 0.0017926	& --
        \\
        1	& 0.0067425	& 0.38	& 0.0032217	& 1.03	& 6.9095e-04	& 2.33	& 1.5377e-04	& 3.54
        \\
        2	& 0.0044515	& 0.59	& 8.9753e-04	& 1.84	& 1.148e-04	& 2.58	& 1.1642e-05	& 3.72
        \\
        3	& 0.0017281	& 1.36	& 1.3976e-04	& 2.68	& 9.2456e-06	& 3.63	& 4.6968e-07	& 4.63
        \\
        4	& 5.7065e-04	& 1.59	& 2.0850e-05	& 2.74	& 6.2609e-07	& 3.88	& 1.6004e-08	& 4.87
        \\
        5	& 1.6469e-04	& 1.79	& 2.8474e-06	& 2.87	& 4.0472e-08	& 3.95	& 5.2234e-10	& 4.93
        \\
        6	& 4.4342e-05	& 1.89	& 3.7210e-07	& 2.93	& 2.5686e-09	& 3.97	& 1.6668e-11	& 4.96
        \\
      \bottomrule
    \end{tabular}
  }
\end{table}

\subsection*{Acknowledgments}

Helpful discussions with Yuri Bazilevs, Andrea Bonito, and Alan Demlow are acknowledged.

\end{document}